\documentclass[leqno,12pt]{amsart} 

\usepackage{amssymb}

\newtheorem{thm}{Theorem}[section]
\newtheorem{prop}[thm]{Proposition}
\newtheorem{lemma}[thm]{Lemma}
\newtheorem{cor}[thm]{Corollary}

\theoremstyle{definition}
\newtheorem{defn}[thm]{Definition}

\theoremstyle{remark}

\numberwithin{equation}{section}

\def\C{\mathbb{C}}
\def\Z{\mathbb{Z}}
\def\R{\mathbb{R}}

\def\Z{\mathbb{Z}}
\def\E{\mathcal{E}}

\def\g{\mathfrak{g}}

\def\O{\mathcal{O}}
\def\D{\mathcal{D}}
\def\F{\mathcal{F}}
\def\CC{\mathcal{C}}

\newcommand{\de}{\partial}
\newcommand{\db}{\overline{\partial}}
\newcommand{\ddb}{{\partial }\overline{\partial}}
\newcommand{\pkk}{$p-$K\"ahler }

\def\L{{\mathcal{L}}}
\def\H{\mathcal{H}}
\def\B{\mathcal{B}}
\def\G{\mathcal{G}}
\def\U{\mathcal{U}}
\def\V{{\mathcal{V}}}
\def\W{\mathcal{W}}
\def\A{\mathcal{A}}
\def\G{\mathcal{G}}
\def\U{\mathcal{U}}

\begin{document}

\title[duality]{Forms and currents defining generalized $p-$K\"ahler structures}

\author{Lucia Alessandrini}
\address{ Dipartimento di Scienze Matematiche Fisiche e Informatiche\newline
Universit\`a di Parma\newline
Parco Area delle Scienze 53/A\newline
I-43124 Parma
 Italy} \email{lucia.alessandrini@unipr.it}



\begin{abstract}
This paper is devoted, first of all, to give a complete unified proof of the Characterization Theorem for compact generalized \pkk  manifolds (Theorem 3.2). The proof is based on the classical duality between \lq\lq closed\rq\rq positive forms and \lq\lq exact\rq\rq positive currents. 
In the last part of the paper we approach the general case of non compact complex manifolds, where  \lq\lq exact\rq\rq positive forms seem to play a more significant role than \lq\lq closed\rq\rq forms.
In this setting, we state the appropriate characterization theorems and give some interesting applications. 

\end{abstract}

\maketitle

\section{Introduction}

In his fundamental work \cite{Su} (1976), D. Sullivan started to study compact complex manifolds using \lq cycles\rq \ and, more generally, positive currents. As he says in the Introduction:

\lq\lq The idea is to consider currents which are \lq directed\rq \   by an a-priori given field of cones in the spaces of tangent $p-$vectors. Such a positivity condition leads to a compact convex cone of currents with a compact convex subcone of cycles (closed currents) (\dots) Moreover, because of the compactness one can apply the basic tools of linear analysis such as the theorems of Hahn-Banach and Choquet. The former allows one to construct closed $\CC^{\infty}-$forms satisfying positivity conditions (on the cone field) because of the duality between forms and currents.\rq\rq 

He observed that a compact complex $n-$dimensional manifold $M$ has natural cone structures, defined by the complex structure $J$: at a point $x$, $C_p(x)$ is the compact (i.e., with compact basis, see Definition I.1 ibidem) convex cone in $\Lambda_{2p}(T'_{x}M)$ generated by the positive combinations of $p-$dimensional complex subspaces. Moreover, a smooth  form 
$\Omega \in \E^{2p}(M)$ is transversal to the cone structure $C_p$ if, for every $x \in M$, and for every $v \in C_p(x), v \neq 0$, it holds $\Omega (v) > 0$
(see Definitions I.3 and I.4 ibidem).

The cone $\CC$ of structure currents associated to the cone structure $C_p$ is the closed convex cone of currents generated by the Dirac currents associated to elements of $C_p(x), x \in M$. In $\CC$, the closed currents are called structure cycles.

Sullivan proved, simply using the Hahn-Banach theorem, that on a compact complex manifold $M$ (Theorem I.7):

a) If no closed transverse form exists, some non-trivial structure cycle is homologous to zero in $M$.

b) If no non-trivial structure cycle exists, some transversal closed form is cohomologous to zero.

He gave some relevant applications: to symplectic structures on a compact complex manifold (sections 10 and 11), and, partially, to compact K\"ahler manifolds (III.15 and III.16).
\medskip

Later on, Harvey and Lawson \cite{HL} (1983) and Michelson \cite{Mi} (1982) apply the same ideas to compact K\"ahler and balanced manifolds, getting an \lq intrinsic characterization\rq \  of K\"ahler and balanced compact manifolds. While Sullivan considered a transversal symplectic 2-form, in duality with null-homologous structure cycles, Harvey and Lawson want to characterize by means of positive currents the K\"ahler condition, i.e. the datum of a closed strictly positive $(1,1)-$form. It turns out that the right space of currents is that of positive currents of bidimension $(1,1)$, which are $(1,1)-$components of boundaries (i.e., $T = (dS)_{1,1}$); such currents are structure currents in the sense of Sullivan, but no more structure cycles! (see \cite{HL}, p. 170). Hence they are no more flat currents, in general, and the closeness of the space of $(1,1)-$components of boundaries has to be proved, to allow the use of a Separation Theorem.

The same considerations apply to $(n-1,n-1)-$components of boundaries in the work of Michelsohn \cite{Mi} and to the case $1 < p< n-1$, which has been studied starting from \cite{AA} (1987), using both closed transverse $(p,p)-$forms (\pkk forms) and closed real transverse  $2p-$forms ($p-$symplectic forms).
\medskip

Some years later, also other \lq\lq closeness\rq\rq conditions on the fundamental forms of hermitian metrics have been studied: in particular, pluriclosed (i.e. closed with respect to the operator $i \ddb$) metrics (see  \cite{Eg}); such metrics are often called  {\it strong K\"ahler metrics with torsion} (SKT) (see among others \cite{FT1} or \cite{FT2}). 
Moreover, $(n-1)-$symplectic metrics are called {\it strongly Gauduchon metrics (sG)} by Popovici (see  \cite{Po1} and \cite{Po2}), while $(n-1)-$pluriclosed metrics are called {\it standard} or {\it Gauduchon metrics}. 

Hence we proposed in \cite{A1} (2011) a unified vision of the whole subject, by introducing for every $p, \ 1 \leq p \leq n-1$, the four classes of {\it generalized \pkk manifolds} (see Section 3).
\medskip

This paper is devoted, first of all, to give a complete unified proof of the Characterization Theorem for compact generalized \pkk  manifolds. The proof is based on the classical duality between \lq\lq closed\rq\rq positive forms and \lq\lq exact\rq\rq positive currents. 

We develop this kind of ideas in the other parts of the paper in two directions: reversing the role of closeness and exactness (\lq\lq closed\rq\rq positive currents and \lq\lq exact\rq\rq positive forms), and approaching the general case of non compact complex  manifolds.

As a matter of fact, the natural environment of \lq\lq exact\rq\rq   \pkk forms is that of non compact manifolds; indeed, $\C^n$ and Stein manifolds are K\"ahler with a form $\omega = i \ddb u$ ($u$ is a smooth strictly plurisubharmonic function).

Moreover, $q-$complete manifolds, and  1-convex manifolds with exceptional set $S$ of dimension $q-1$, are $p-$K\"ahler for ever $p \geq q$, with a $\ddb-$exact form.

Thus in Section 8 we state the convenient characterization theorems in the non compact case and give some interesting applications. 
\medskip

The plane of the paper is as follows:

In Section 2 we discuss the notion of positivity of forms, vectors and currents, while in Section 3 we introduce the generalized \pkk manifolds and their characterization by \lq\lq exact\rq\rq positive currents in the compact case. The complete proof of the Characterization Theorem 3.2 is given in Section 4, where we introduce also the machinery of exact sequences of suitable sheaves, that we shall use also in the second part of the paper.

In Section 5 we propose a characterization theorem with \lq\lq closed\rq\rq  currents and \lq\lq exact\rq\rq  forms on compact manifolds, also inspired by the work of Sullivan, which makes sense for $p > 1$.

From Section 6 on, we try to put exact generalized \pkk forms, or also  \lq locally\rq \  generalized closed  \pkk forms, on some classes of non compact manifolds. We collect in Section 7 the machinery to get some information about Bott-Chern and Aeppli cohomology, and in Section 8 the characterization theorems on non compact manifolds.  
\bigskip

 \section{Basic tools}
 
 Let $X$ be a complex manifold of dimension $n \geq 2$, let $p$ be an integer, $1 \leq p \leq n$. 
The purpose of this section is to discuss positivity of $(p,p)-$forms, $(p,p)-$vectors and $(p,p)-$currents: we refer to \cite{HK} and to \cite{De} as regards notation and terminology. 
 
 Positivity  involves only multi-linear algebra; therefore, take a complex vector space $E$ of dimension $n$, its associated vector spaces of $(p,q)-$forms  $\Lambda^{p,q}(E^*)$, 
 and a basis $\{\varphi_1, \dots, \varphi_n \}$ for $E^*$. 
   
 Let us  denote by $\varphi_I$ the product $\varphi_{i_1} \wedge \dots \wedge \varphi_{i_p}$, where $I = (i_1, \dots, i_p)$ is an increasing multi-index.
 Call $\sigma_p := i^{p^2} 2^{-p}$; thus,  
if $\zeta, \eta \in \Lambda^{p,0} (E^*)$, then $\overline{\sigma_p \zeta \wedge \bar{\eta}} = \sigma_p \eta \wedge \bar{\zeta}$, so that $\sigma_p \eta \wedge \bar{\eta}$ is real;
hence we get obviously that
 $\{ \sigma_p \varphi_I \wedge \overline{\varphi_I} , |I| = p \}$ is a basis for $\Lambda ^{p,p}_{\R} (E^*):= \{ \Psi \in \Lambda ^{p,p} (E^*) / \Psi = \overline{\Psi} \}$ and
 that $$dv = (\frac{i}{2}  \varphi_1 \wedge \overline{\varphi_1}) \wedge \dots \wedge (\frac{i}{2}  \varphi_n \wedge \overline{\varphi_n}) = \sigma_n \varphi_I \wedge \overline{\varphi_I} , \ I=(1, \dots , n)$$
is a volume form. 
We call a $(n,n)-$form $\tau$ {\it positive} ({\it strictly positive}) if $\tau = c\ dv, \ c \geq 0 \ ( c>0)$. We shall write $\tau \geq 0 \ ( \tau > 0)$.

\medskip

From now on, let $1 \leq p \leq n-1$ and let $k := n-p$.

\begin{defn}   
\begin{enumerate}
\item $\eta \in \Lambda^{p,0} (E^*)$ is called {\it simple} (or decomposable) if and only if there are $\{\psi_1, \dots, \psi_p \} \in E^*$ such that $\eta = \psi_{1} \wedge \dots \wedge \psi_{p}$.

\item $\Omega \in \Lambda ^{p,p}_{\R} (E^*)$ is called {\it strongly positive} ($\Omega \in SP^p$) if and only if 
$ \Omega = \sigma_p \sum_j \eta_j \wedge \overline{\eta_j} ,$ with $\eta_j$ simple.

\item $\Omega \in \Lambda ^{p,p}_{\R} (E^*)$ is called {\it positive} ($\Omega \in P^p$) if and only if for all $\eta \in \Lambda^{k,0} (E^*)$,
the $(n,n)-$form
$\tau := \ \Omega \wedge \sigma_k \eta \wedge \overline{\eta}$ is positive.

\item  $\Omega \in \Lambda ^{p,p}_{\R} (E^*)$ is called {\it weakly positive} ($\Omega \in WP^p$) if and only if 
for all $\psi_j \in E^*$, and for all $I = (i_1, \dots, i_k)$,
$\Omega \wedge \sigma_k \psi_I \wedge \overline{\psi_I}$ is a positive $(n,n)-$form. It is called {\it transverse} when it is {\it strictly} weakly positive, that is, when $\Omega \wedge \sigma_k \psi_I \wedge \overline{\psi_I}$ is a strictly positive $(n,n)-$form for $\sigma_k \psi_I \wedge \overline{\psi_I} \neq 0$ (i.e. $\psi_{i_1} ,\dots , \psi_{i_k}$ linearly independent).
\end{enumerate}
\end{defn}
\medskip

 {\bf 2.1.1 Remarks.}
 \medskip
 
 a) The sets $P^p, SP^p, WP^p$ and their interior parts are indeed convex cones; moreover, there are obvious inclusions: 
 $ SP^p \subseteq P^p \subseteq WP^p \subseteq \Lambda ^{p,p}_{\R}(E^*)$
 \medskip
 
 b) When $p=1$ or $p=n-1$, the three cones coincide, since every $(1,0)-$form is simple (and hence also every $(n-1,0)-$form is simple).
  In the intermediate cases, $1< p< n-1$, the inclusions are strict (\cite{HK}).
   \medskip
 
 c) Using the volume form $dv$,  we get the pairing 
 $$f : \Lambda ^{p,p}(E^*) \times \Lambda ^{k,k} (E^*) \to \C, \ \ f(\Omega, \Psi)dv = \Omega \wedge \Psi .$$ 

 Thus:   $$\Omega \in WP^p \iff \forall \ \Psi \in SP^k, \Omega \wedge \Psi \geq 0.$$
  $$\Omega \in SP^p \iff \forall \ \Psi \in WP^k, \Omega \wedge \Psi \geq 0,$$
 $$ \Omega \in P^p \iff \forall \ \Psi \in P^k, \Omega \wedge \Psi \geq 0.$$

\bigskip
As regards vectors, consider $\Lambda_{p,q}(E)$, the space of $(p,q)-$vectors: as before, $V \in \Lambda_{p,0}(E)$ is called a {\it simple vector} if 
$V = v_{1} \wedge \dots \wedge v_{p}$ for some $v_j \in E$; 
in this case, when $V \neq 0$, $\sigma_p^{-1} V \wedge \overline V$ is called a {\it strictly strongly positive} $(p,p)-$vector.
We can identify strictly strongly positive $(p,p)-$vectors  with $p-$planes in $\C^n$, i.e. with the elements of $G_{\C}(p,n)$; to every plane corresponds a unique unit vector.

\medskip
  \begin{prop}
$\Omega \in \Lambda ^{p,p}_{\R} (E^*)$ is  transverse if and only if $\Omega(\sigma_p^{-1} V \wedge \overline V) > 0$
for every $V \in \Lambda_{p,0}(E)$, $V \neq 0$ and simple.
  \end{prop}

\begin{proof}  
Using the pairing $f$, we get an isomorphism $g: \Lambda_{p,p}(E) \to \Lambda ^{k,k} (E^*)$ given as:
$f(\Omega, g(A)) = \Omega(A),$ i.e.
$$f(\Omega, g(A)) dv = \Omega \wedge g(A) := \Omega(A) dv, \quad  \forall A \in \Lambda_{p,p}(E), \forall \Omega \in  \Lambda ^{p,p} (E^*).$$
If $\{e_1, \dots, e_n \}$ denotes the dual basis of $\{\varphi_1, \dots, \varphi_n \}$, it is not hard to check that for all $I = (i_1, \dots, i_p)$, 
$g(\sigma_p^{-1} e_I \wedge \overline {e_I}) = \sigma_k \varphi_J \wedge \overline {\varphi_J}$ with $J =  \{1, \dots, n \} - I.$

Thus 
the isomorphism $g$ transforms $(p,p)-$vectors of the form $\sigma_p^{-1} V \wedge \overline V$, with $V$ simple,  into strongly positive $(k,k)-$forms (of the form $\sigma_k \eta \wedge \overline{\eta},$ with $\eta$ simple). 
Hence we get
$$\Omega(\sigma_p^{-1} V \wedge \overline V) dv = \Omega \wedge g(\sigma_p^{-1} V \wedge \overline V) =  \Omega \wedge \sigma_k \eta \wedge \overline \eta $$
and the statement follows. 
\end{proof}
 
 \bigskip 

Let us turn back to a manifold $X$; for $0 \leq p \leq n$, we denote  by ${\D}^{p,p}(X)_\R$ the space of compactly supported real $(p,p)-$forms on $X$ and by ${\E}^{p,p}(X)_\R$ the space of real $(p,p)-$forms on $X$. 

Their dual spaces are: ${\D}_{p,p}'(X)_{\R}$ (also denoted by ${\D '}^{k,k}(X)_{\R}$, where $p+k=n$), the space of real currents of bidimension $(p,p)$ or bidegree $(k,k)$, which we call {\it $(k,k)-$currents}, and 
${\E}_{p,p}'(X)_{\R}$ (also denoted by ${\E '}^{k,k}(X)_{\R}$), the space of compactly supported real $(k,k)-$cur\-rents on $X$.

\begin{defn}  
 The form $\Omega \in {\E}^{p,p}(X)_\R$ is called {\it strongly positive} (resp. {\it positive, weakly positive, transverse} or {\it strictly weakly positive}) if: 
 
 $\forall \ x \in X, \ \Omega_x \in SP^p (T'_xX^*)$ (resp. $P^p (T'_xX^*), \  WP^p (T'_xX^*),$ $(WP^p (T'_xX^*))^{int}$). 

These spaces of forms are denoted by
$SP^p(X), \ P^p(X),$ $ WP^p(X),$ $(WP^p(X))^{int}$. 
\end{defn}

\begin{defn}  
Let $T \in {\E}_{p,p}'(X)_{\R}$ be a current of bidimension $(p,p)$ on $X$. Let us define:

{\it weakly positive currents}: $T \in WP_p(X) \iff T(\Omega) \geq 0\ \ \forall \ \Omega \in SP^p(X)$. 

{\it positive currents}: $T \in P_p(X) \iff T(\Omega) \geq 0\ \ \forall \ \Omega \in P^p(X)$. 

{\it strongly positive currents}: $T \in SP_p(X) \iff T(\Omega) \geq 0\ \ \forall \ \Omega \in WP^p(X)$. 
\end{defn}
\medskip

 {\bf Notation.} $\Omega \geq 0$ denotes that $\Omega$ is weakly positive;  $\Omega > 0$ denotes that $\Omega$ is transverse; $T \geq 0$ means that $T$ is strongly positive. Thus:
 
 {\bf 2.4.1 Claim.} $\Omega > 0$  if and only if $T(\Omega) > 0$ for every $T \geq 0, T \neq 0.$
 \medskip

 {\bf Remarks.}  Obviously the previous cones of currents satisfy:
 $SP_p(X) \subseteq  P_p(X) \subseteq WP_p(X)$. The classical positivity for currents (i.e. \lq\lq positive in the sense of Lelong\rq\rq) is strong positivity; Demailly (\cite{De}, Definition III.1.13) does not consider $P_p(X)$, and indicates $WP_p(X)$ as the cone of positive currents; there is no uniformity of notation in the papers of Alessandrini and Bassanelli.

Moreover, let us recall that, if  $f$ is a holomorphic map, and  $T \geq 0$, then  $f_* T \geq 0$.

 \bigskip
 
The differential operators $d, \de, \db$ extends naturally to currents by duality; thus we have two De Rham complexes, $(\E^*, d)$ and $((\D')^*, d)$; but the embedding $i: (\E^*, d) \to ((\D')^*, d)$ induces an isomorphism at the cohomology  level. This fact applies also to other cohomologies, as  Bott-Chern and  Aeppli. Since the notation has changed during the last 50 years, we recall them below:

$$H_{\ddb}^{k,k}(X, \R) = \Lambda_\R ^{k,k}(X) = H_{BC}^{k,k}(X, \R) :=\frac{\{ \varphi \in {\E}^{k,k}(X)_\R;
d\varphi =0\}}{\{i\partial\overline{\partial}\alpha ;\alpha \in {\E}^{k-1,k-1}(X)_\R\}}$$
$$H_{\de + \db}^{k,k}(X, \R) =V_\R ^{k,k}(X) = H_{A}^{k,k}(X, \R) :=\frac{\{ \varphi \in {\E}^{k,k}(X)_\R;
i\ddb\varphi =0\}}{\{\varphi = \de \overline\alpha + \db \alpha ; \alpha \in {\E}^{k,k-1}(X)\}}$$
$$H_{d} ^{k,k}(X, \R) :=\frac{\{ \varphi \in {\E}^{k,k}(X)_\R;
d\varphi =0\}}{\{  \varphi \in {\E}^{k,k}(X)_\R; \varphi = d\eta; \eta \in {\E}^{2k-1}(X)_\R\}}$$
$$H_{DR}^{j}(X, \R) :=\frac{\{ \zeta \in {\E}^{j}(X)_\R;
d\zeta =0\}}{\{  \zeta \in {\E}^{j}(X)_\R; \zeta = d\eta; \eta \in {\E}^{j-1}(X)_\R\}} .$$

\medskip
In general, when the class of a form or a current vanishes in one of the previous cohomology groups, we say that the form or the current \lq\lq bounds\rq\rq or is \lq\lq exact\rq\rq. 
\bigskip

  \section{Generalized $p-$K\"ahler conditions on compact manifolds}

We introduced  $p-$K\"ahler manifolds in \cite{AA} and then in \cite{AB1}, and studied them mainly in the compact case:
$p-$K\"ahler manifolds enclose K\"ahler and balanced manifolds, and seem to be the better generalization of the K\"ahler setting. Later on, also pluriclosed (SKT) manifolds have been proposed as a good generalization of K\"ahler manifolds.

Thus a deep investigation of this type of structures (no more metrics, in general) was needed: we proposed in \cite{A1} a general setting, those of 
{\it generalized $p-$K\"ahler manifolds}, which enclose all the known classes of non-K\"ahler manifolds that can be characterized by a transverse 
\lq\lq closed\rq\rq form. In the last years, some of them have been studied (not with the same name!) by other authors: hence we give in Remark 3.1.2  a sort of  dictionary; moreover, a brief survey of the whole history can be seen looking at the proofs of the suitable Characterization Theorems, as we indicate in the Remarks after Theorem 3.2.

\medskip

 \begin{defn} Let $X$ be a complex manifold of dimension $n \geq 2$, let $p$ be an integer, $1 \leq p \leq n-1$.
 
\begin{enumerate}
\item $X$ is a {\it $p-$K\"ahler (pK) manifold} if it has a closed transverse (i.e. strictly weakly positive) $(p,p)-$form $\Omega \in \E^{p,p}(X)_{\R}$. 

\item $X$ is a {\it weakly $p-$K\"ahler (pWK) manifold} if it has a transverse $(p,p)-$form $\Omega$ with $\de \Omega = \ddb \alpha$ for some form $\alpha$.

\item $X$ is a {\it $p-$symplectic (pS) manifold} if it has a closed transverse  real $2p-$form $\Psi \in \E^{2p}(X)$; that is, $d \Psi = 0$ and $\Omega := \Psi^{p,p}$ (the  $(p,p)-$component of $\Psi$) is transverse.

\item  $X$ is a {\it $p-$pluriclosed  (pPL) manifold} if it has a transverse $(p,p)-$form $\Omega$ with $\ddb \Omega = 0.$
\end{enumerate}

\end{defn}

Notice that:
$pK \Longrightarrow pWK  \Longrightarrow pS  \Longrightarrow pPL;$ 
as regards examples and differences under these classes of manifolds, see \cite{A1}, \cite{A2}, \cite{A3}.

When $X$ satisfies one of these definitions, it is called a {\bf generalized $p-$K\"ahler manifold}; the form $\Omega$, called a generalized  \pkk form, is said to be \lq\lq closed\rq\rq.

\medskip

{\bf 3.1.1 Remark. } As regards Definition 3.1(3), let us write the condition $d \Psi = 0$ in terms of a condition on $\de \Omega$, as in the other statements; 
when $\Psi = \sum_{a+b=2p} \Psi^{a,b},$ then $d \Psi = 0$ is  equivalent to:

i) $\db \Psi^{n-j,2p-n+j} + \de \Psi^{n-j-1,2p-n+j+1}=0$, for $j=0, \dots, n-p-1$, when $n\leq 2p$

and

ii) $\de \Psi^{2p,0}=0, \ \db \Psi^{2p-j,j} + \de \Psi^{2p-j-1,j+1}=0$, for $j=0, \dots, p-1$, when $n > 2p.$

In particular, $\de \Omega = \de \Psi^{p,p}= - \db \Psi^{p+1,p-1}$ (which is the only condition when $p=n-1$, as remarked also in \cite{Po1}).
\medskip

{\bf 3.1.2 Remark. } 1PL corresponds to pluriclosed (\cite{Eg}) or SKT (\cite{FT2}); 1S to hermitian symplectic (\cite{ST}), 1K to K\"ahler.
Moreover, $(n-1)$PL manifolds (or metrics) are called standard or Gauduchon; $(n-1)$S corresponds to strongly Gauduchon (\cite{Po1}, \cite{X}), $(n-1)$WK manifolds are called superstrong Gauduchon (\cite{PU}), $(n-1)$K corresponds to balanced (\cite{Mi}).

\medskip

Let us go to the Characterization Theorem. As in the work of Harvey and Lawson \cite{HL}, some questions arise about the natural operators as $i \ddb, d, \de + \db$: do they have closed range? Let us recall how the problem is solved in \cite{HL} when $M$ is compact, to emphasize the crucial points of the general case. The authors prove in Section 2 that, when $M$ is compact:

\begin{enumerate}
\item For every $p$, $dim H^p(M, \H) < \infty$, where $\H$ is the sheaf of germs of pluriharmonic functions; this is due to the finite dimensionality of $H^j(M, \R)$ and $H^j(M, \O)$, using the exact sequence (4.11) in Section 4.

\item The image of $d : {\E}^{1,1}(M)_\R \to {Z}^{1,1}(M)_\R = \{ \psi \in ({\E}^{2,1}(M) \oplus {\E}^{1,2}(M))_\R / d \psi = 0 \}$ has 
finite codimension in ${Z}^{1,1}(M)_\R$, because $H^2(M, \H)  \simeq {Z}^{1,1}(M)_\R / d {\E}^{1,1}(M)_\R$. This fact is due to the cohomology sequences coming from the exact sequence of sheaves (4.1) in Section 4.

\item The operator $d : {\E}^{1,1}(M)_\R \to ({\E}^{2,1}(M) \oplus {\E}^{1,2}(M))_\R$ has closed range, since by (2) the image of $d$ is closed in ${Z}^{1,1}(M)_\R$.

\item On currents, let $\pi : {\E}_{2}'(M)_{\R} \to {\E}_{1,1}'(M)_{\R}$ be the natural projection; the operator 
$d_{1,1} : ({\E}_{2,1}'(M) \oplus {\E}_{1,2}'(M))_\R \to {\E}_{1,1}'(M)_\R$ given by $d_{1,1} = \pi \circ d$ restricted to
$({\E}_{2,1}'(M) \oplus {\E}_{1,2}'(M))_\R$, is the adjoint operator to  $d : {\E}^{1,1}(M)_\R \to ({\E}^{2,1}(M) \oplus {\E}^{1,2}(M))_\R$, so that it has closed range (see (4.8) and Section 4).

\item Thus $Im d_{1,1}$, that is, the space of currents which are $(1,1)-$components of a boundary, is closed in ${\E}_{1,1}'(M)_{\R}$.
\end{enumerate}

\medskip
We shall develop these steps  to get the proof of the general Characterization Theorem. Thus, in the same vein, we prove:

\begin{thm} 

\begin{enumerate}
\item {\bf Characterization of compact $p-$K\"ahler (pK) manifolds.}

$M$ has a strictly weakly positive $(p,p)-$form $\Omega$ with $\de \Omega = 0$,  if and only if $M$ has no strongly positive currents $T \neq 0$, of bidimension $(p,p)$, such that $T = \de  \overline S + \db S$ for some current $S$ of bidimension $(p,p+1)$ (i.e.  $T$  is the $(p,p)-$component of a boundary). 

\item {\bf Characterization of compact weakly $p-$K\"ahler (pWK) manifolds.}

$M$ has a strictly weakly positive  $(p,p)-$form $\Omega$ with $\de \Omega = \ddb \alpha$ for some form $\alpha$,  if and only if $M$ has no strongly  positive currents $T \neq 0$, of bidimension $(p,p)$, such that $T = \de  \overline S + \db S$ for some current $S$ of bidimension $(p,p+1)$ with $\ddb S = 0$ (i.e.  $T$  is closed and is the $(p,p)-$component of a boundary). 

\item {\bf Characterization of compact $p-$symplectic (pS) manifolds.}

$M$ has a real $2p-$form $\Psi = \sum_{a+b=2p} \Psi^{a,b}$, such that $d \Psi = 0$ and the 
$(p,p)-$form $\Omega := \Psi^{p,p}$ is 
strictly weakly positive,   if and only if $M$ has no strongly  positive currents $T \neq 0$, of bidimension $(p,p)$, such that $T = d R$ for some current $R$  (i.e.  $T$  is a boundary, that is, $T = \de  \overline S + \db S$ with $\de S = 0$). 

\item {\bf Characterization of  compact $p-$pluriclosed (pPL) manifolds.}

$M$ has a strictly weakly positive $(p,p)-$form $\Omega$ with $\ddb \Omega = 0$,  if and only if $M$ has no strongly  positive currents $T \neq 0$, of bidimension $(p,p)$, such that $T = \ddb A$ for some current $A$ of bidimension $(p+1,p+1)$. 
\end{enumerate}
\end{thm}

{\bf Remarks.}  

Theorem 3.2(1) for $p=1$ was proved  in \cite{HL}, Theorem 14;

Theorem 3.2(1) for $p=n-1$ was proved   in \cite{Mi}, Theorem 4.7;

Theorem 3.2(1) for a generic $p$ was proved   in \cite{AA}, Theorem 1.17;

Theorem 3.2(2) for $p=1$ was proved in \cite{HL}, Theorem 38; in fact, Theorem 3.2(2) is related to a question posed by Harwey and Lawson in their paper (Section 5 in \cite{HL}), about the use of {\it closed} currents in  characterization theorems (this is important because closed positive currents are flat in the sense of Federer).

Theorem 3.2(3)  for $p=1$ was proved in \cite{Su}, Theorems III.2 and III.11;

Theorem 3.2(3) for a generic $p$ was proved   in \cite{AA}, Theorem 1.17;

Theorem 3.2(3) for $p=n-1$ is proved also in \cite{Po1}, Proposition 3.3.

Theorem 3.2(4) for $p=1$ is proved in \cite{Eg}, Theorem 3.3.

Recall also a result of Gauduchon (\cite{Ga}) (for $p=n-1$), who proved that every compact $n-$dimensional manifold is $(n-1)PL$.
As a matter of fact, this result is now a corollary of the previous Theorem, since for $p=n-1$, the current $A$ in Theorem 3.2(4) reduces to a plurisubharmonic global function on a compact complex manifold, hence to a constant. Such a metric is also called a standard (or Gauduchon) metric.

We shall give a complete proof of all statements in the next section.
\bigskip

  \section{Proof of the Characterization Theorem 3.2}

Let us firstly recall some well-known facts about Fr\'echet topological vector spaces 
and Fr\'echet sheaves that we shall use here and in Section 5 and 8.

\begin{lemma} {\rm (see  \cite{Sc} IV.7.7)} Let $L , M$ be Fr\'echet 
spaces, and let $f : L \to M$ be a continuous linear map. Then $f$ is a
topological homomorphism if and only if $f$ has closed range, that is, 
if and only if  $M \over {Im f}$ is a Hausdorff (hence Fr\'echet) t.v.s. 
\end{lemma}

\begin{lemma} {\rm (see  \cite{Se} page 21)} Let $L , M$ be Fr\'echet 
spaces, and let $f : L \to M$ be a continuous linear map whose image has 
finite codimension. Then $f$ is a topological homomorphism
 (i.e. $M \over {Im f}$ is Hausdorff,  i.e. $Im f$ is closed in $M$).
\end{lemma}

\begin{lemma} Let $L , M$ be Fr\'echet 
spaces, let $f : L \to M$ be a continuous surjective linear map. Let $N$ be a 
closed subspace of $L$ with finite codimension. Then $f(N)$ is closed.
\end{lemma}

\begin{proof} Consider the induced map $g : {L\over N}  \to {M \over {f(N)}}$
which is surjective: hence $f(N)$ has finite codimension in $M$. Now $N$ is a Fr\'echet space,
 and $f\vert _N : N \to M$ satisfies Lemma 4.2, thus $ M \over {f(N)}$ is Hausdorff.
\end{proof}

\begin{thm} {\rm (Hahn-Banach Theorem, see \cite{Sc}, Theorem II.3.1)} Let $E$ be a topological vector space, let $F$ be a linear manifold in $E$, and let $A$ be a non-empty convex open subset of $E$, not intersecting $F$. There exists a closed hyperplane in $E$, containing $F$ and not intersecting $A$.
\end{thm}

\begin{thm} {\rm Separation  Theorem, see \cite{Sc}, Theorem II.9.2)} Let $E$ be a locally convex space, let $A,B$  non-empty disjoint convex subsets of $E$, such that $A$ is closed and $B$ is compact. There exists a closed hyperplane in $E$, strictly separating $A$ and $B$.
\end{thm}

\medskip
Let us go now to the preliminaries of the proof of Theorem 3.2. 
\medskip

Let $X$ be a complex $n-$dimensional manifold;
for $n \geq p,q \geq 0$, consider
the spaces ${\E}^{p,q}(X)$, endowed with the usual topology of the uniform convergence on compact sets: they are Fr\'echet spaces. Their topological dual
spaces  (with the weak topology) are the spaces ${\E}_{p,q}'(X)$, and the pairing is denoted by $S(\alpha)$ or $(S, \alpha)$ for every 
$S \in {\E}_{p,q}'(X)$ and $\alpha \in {\E}^{p,q}(X)$. If $F \subset {\E}^{p,q}(X)$, $S \in F^{\perp}$ means that $(S, \alpha)=0$ for all $\alpha \in F$.

Moreover, we denote as usual by  ${\E}^{p,q}_\R$ the sheaf of germs of real $(p,q)-$forms, and by $\Omega^{j}$ the sheaf of germs of holomorphic $j-$forms.
\medskip

Notice that in \cite{HL} (Proposition 1) only the following resolution of the sheaf $\H$ is needed:
 \begin{equation}\label{1}
 0 \to \H \stackrel{j}{\to}  {\E}^{0,0}_\R   \stackrel{i \ddb}{\to}  {\E}^{1,1}_\R \stackrel{d}{\to} ({\E}^{2,1} \oplus {\E}^{1,2})_\R  
 \stackrel{d}{\to} {\E}^{4}(M)_\R \to \dots 
 \end{equation}
where $j$ is the standard inclusion.  On the contrary, our situation  is much more complicated, because it involves in the resolution of $\H$ also sheaves whose cohomology is not trivial (see f.i. \cite{ABL}, page 259; the notation stems mainly from \cite{Bi1} and \cite{Bi2} ).
\bigskip

We consider the following resolution of the sheaf $\H$, for $p > 0$:
 
  \begin{equation}\label{2}
0 \to \H \stackrel{\sigma_{-1}}{\to}  {\L}^{0}   \stackrel{\sigma_{0}}{\to}  \dots    {\L}^{p-1}   \stackrel{\sigma_{p-1}}{\to} {\B}^{p} \stackrel{\sigma_{p}}{\to} \dots {\B}^{2p-1} \stackrel{\sigma_{2p-1}}{\to} 
  \end{equation}
$${\E}^{p,p}_\R   \stackrel{\sigma_{2p}}{\to}  {\E}^{p+1,p+1}_\R \stackrel{\sigma_{2p+1}}{\to} ({\E}^{p+2,p+1} \oplus {\E}^{p+1,p+2})_\R  
 \stackrel{\sigma_{2p+2}}{\to} {\E}^{2p+4}_\R \to \dots .$$

Here, 

$\L^j := (\Omega^{j+1} \oplus ( \oplus_{k=0}^{j} \E^{j-k,k}) \oplus \overline  \Omega^{j+1})_\R$, for $0 \leq j \leq p-1$;

$\B^j := ( \oplus_{k=0}^{2p-j} \E^{p-k, j-p+k})_\R$, for $p \leq j \leq 2p-1$

\noindent and the maps are, respectively,

\begin{enumerate}
\item $\sigma_{-1}(h) = (- \de h , h, - \db h)$, 
 
\item for $0 \leq j \leq p-2$ (if $p > 1$):

$\sigma_{j} ( \varphi^{j+1}, \{\alpha^{j-k,k}\}_{k=0}^{j}, \overline \varphi^{j+1}) =$ 

$(- \de \varphi^{j+1}, \de \alpha^{j,0} + \varphi^{j+1}, \{\db \alpha^{k+1,j-k-1} + \de \alpha^{k,j-k}\}_{k=0}^{j-1}, \overline \varphi^{j+1} + 
\db \alpha^{0,j}, - \db \overline \varphi^{j+1})$;

\item $\sigma_{p-1} (\varphi^{p}, \{\alpha^{p-1-k,k}\}_{k=0}^{p-1}, \overline  \varphi^{p}) 
=$ 

$(\de \alpha^{p-1,0} + \varphi^{p}, \{\db \alpha^{p-1-k,k} + \de \alpha^{p-2-k,k+1}\}_{k=0}^{p-2}, \overline \varphi^{p} + 
\db \alpha^{0,p-1})$;
  
\item for $p \leq j \leq 2p-1$:
 $\sigma_{j} (\{\alpha^{p-k,j-p+k}\}_{k=0}^{2p-j}) = ( \{\db \alpha^{p-k,j-p+k} + \de \alpha^{p-k-1,j-p+k+1}\}_{k=0}^{2p-j-1})$;

\noindent (in particular, $\sigma_{2p-1}(\beta, \overline \beta) = (\db \beta + \ \de \overline \beta)$)

\item $\sigma_{2p} = i \ddb$;

\item $\sigma_{2p+1} = \sigma_{2p+2} = d$.
\end{enumerate}

\smallskip
Moreover, we shall denote by $d_{s}$ the operator $d$ acting on $s-$forms: ${\E}^{s}_\R   \stackrel{d_{s}}{\to}  {\E}^{s+1}_\R$.
\medskip

For instance, when $p = 1$, the exact sequence of sheaves (4.2) becomes
  \begin{equation}\label{3}
  0 \to \H \stackrel{\sigma_{-1}}{\to}  {\L}^{0}   \stackrel{\sigma_{0}}{\to}   {\B}^{1}  \stackrel{\sigma_{1}}{\to} 
  {\E}^{1,1}_\R   \stackrel{\sigma_{2}}{\to}  {\E}^{2,2}_\R \stackrel{\sigma_{3}}{\to} ({\E}^{3,2} \oplus {\E}^{2,3})_\R  
 \to \dots 
 \end{equation}
 
i.e.
  \begin{equation}\label{4}
  0 \to \H \stackrel{\sigma_{-1}}{\to}    ( \Omega^1 \oplus \E^{0,0} \oplus \overline \Omega^1)_{\R} \stackrel{\sigma_{0}}{\to} (\E^{1,0} \oplus \E^{0,1})_{\R} \stackrel{\sigma_{1}}{\to} \E^{1,1}_{\R} \stackrel{\sigma_{2}}{\to} \E^{2,2}_{\R} \stackrel{\sigma_{3}}{\to} ({\E}^{3,2} \oplus {\E}^{2,3})_\R \dots
 \end{equation}
 
 where the maps are, respectively,

$\sigma_{-1}(h) = (- \de h , h, - \db h)$, 
$\sigma_{0}(\varphi, f, \overline \varphi,) = ( \varphi + \de f , \db f + \overline \varphi),$ 

$ \sigma_{1}(\beta, \overline \beta) = (\db \beta + \ \de \overline \beta), \ \ \sigma_{2} = i \ddb \ \ \sigma_{3} = d$, and so on. 
\medskip
 
 When $p=0$, we shall use the sequence (4.1).
\medskip

At the level of sections, we have the following operators:

  \begin{equation}\label{3}
{\E}^{q,q} (X)_\R   \stackrel{\sigma_{2q}}{\to}  {\E}^{q+1,q+1} (X)_\R \stackrel{\sigma_{2q+1}}{\to} ({\E}^{q+2,q+1} \oplus {\E}^{q+1,q+2}) (X)_\R, 
\end{equation}
  \begin{equation}\label{4} (\E^{p,p-1} \oplus \E^{p-1,p}) (X)_\R \stackrel{\sigma_{2p-1}}{\to} \E^{p,p} (X)_\R \stackrel{\sigma_{2p}}{\to} \E^{p+1,p+1} (X)_\R
\end{equation}
 \begin{equation}\label{5} {\E}^{2p-1} (X)_\R \stackrel{d_{2p-1}}{\to}  {\E}^{2p} (X)_\R  \stackrel{d_{2p}}{\to}  {\E}^{2p+1} (X)_\R 
\end{equation}

and their dual operators, acting on currents:
  \begin{equation}\label{6}
 ({\E}_{q+2,q+1}' \oplus {\E}_{q+1,q+2}') (X)_\R  \stackrel{\sigma_{2q+1}'}{\to}  ({\E}_{q+1,q+1}') (X)_\R \stackrel{\sigma_{2q}'}{\to}  ({\E}_{q,q}') (X)_\R, 
\end{equation}
  \begin{equation}\label{7} (\E_{p+1,p+1}') (X)_\R \stackrel{\sigma_{2p}'}{\to} (\E_{p,p}') (X)_\R \stackrel{\sigma_{2p-1}'}{\to}  (\E_{p,p-1}' \oplus \E_{p-1,p}') (X)_\R
\end{equation}
  \begin{equation}\label{8} ({\E}_{2p+1}') (X)_\R \stackrel{d_{2p}'}{\to} ({\E}_{2p}') (X)_\R \stackrel{d_{2p-1}'}{\to}  ({\E}_{2p-1}') (X)_\R.
\end{equation}
\medskip
 
Using this notation, we can rewrite the statements of Theorem 3.2 in a useful manner: f.i., in statement (1), the condition $\de \Omega =0$ (which implies $d \Omega =0$ since $\Omega$ is real), means $\Omega \in  Ker \sigma_{2q+1}$ (where $q := p-1$) and, on the other hand, $T = \de  \overline S + \db S$ means $T \in Im \sigma_{2q+1}'$.
In (2), it is not hard to check that the condition $\de \Omega = \ddb \alpha$ means that 
$\Omega \in (Ker \sigma_{2q+1} + Im \sigma_{2p-1})$, while the condition on $T$ (closed and the $(p,p)-$component of a boundary) means exactly that 
$T \in (Im \sigma_{2q+1}' \cap Ker \sigma_{2p-1}')$, and so on. 

{\bf Warning!} It is important to use the index $q=p-1$, because in this manner we denote in a unified way the bi-degree of the form and the {\it right} operator in the statement of the theorem. For instance, in (1) $\Omega \in  Ker \sigma_{2q+1}$ means the $\Omega$ is a closed real $(p,p)-$form. This is more clear in (2): $T \in Im \sigma_{2q+1}'$ means $T \in \E_{p,p}' (X)_\R, T = \de  \overline S + \db S$, and $T \in Ker \sigma_{2p-1}'$ means 
$T \in \E_{p,p}' (X)_\R, (\de T, \db T)=0$.

\medskip
Summing up, we get the following version of Theorem 3.2: 
\medskip 

{\bf Theorem (3.2)'.}
Let $M$ be a compact complex manifold of dimension $n \geq 2$; let $1 \leq p \leq n-1$ and denote $q := p-1$ in the subscript of the operators
cited in the sequences (4.5) - (4.10). 

\begin{enumerate}

\item There is a real transverse $(p,p)-$form $\Omega$ on $M$ such that $\Omega \in Ker \sigma_{2q+1} \ \iff$    
there are no non trivial currents $T \in {\E}_{p,p}'(X)_{\R}$, $T \geq 0$, $T \in Im \sigma_{2q+1}'.$ 

\item There is a real transverse $(p,p)-$form $\Omega$ on $M$ such that $\Omega \in (Ker \sigma_{2q+1} + Im \sigma_{2p-1}) \ \iff$    
there are no non trivial currents $T \in {\E}_{p,p}'(X)_{\R}$, $T \geq 0$, $T \in (Ker \sigma_{2p-1}' \cap Im \sigma_{2q+1}').$ 

\item There is a real $2p-$form $\Psi$ with $\Psi^{p,p} := \Omega$ transverse
 on $M$ such that $\Psi \in Ker d_{2p} \ \iff $ 
there are no non trivial currents $T \in {\E}_{p,p}'(X)_{\R}$, $T \geq 0$, $T \in Im d_{2p}'.$ 

\item There is a  real transverse $(p,p)-$form $\Omega$ on $M$ such that $\Omega \in Ker \sigma_{2p} \ \iff $
there are no non trivial currents $T \in {\E}_{p,p}'(X)_{\R}$, $T \geq 0$, $T \in Im \sigma_{2p}'.$ 
\end{enumerate}

\medskip

\begin{proof}
(of Theorem 3.2 or (3.2)').  In all cases, one part of the proof is simple: if there exist both the form $\Omega$ and  the current $T$ as given in the Theorem, we would have by Claim 2.4.1:
$(T, \Omega)> 0$. 
But:

Case (1): 
$(T, \Omega)=0$  because $\Omega \in Ker \sigma_{2q+1}$ and $T \in Im \sigma_{2q+1}' \subseteq (Ker \sigma_{2q+1})^{\perp}$.

Case (2): $(T, \Omega)= ( \de  \overline S + \db S, \Omega) = (\overline S,  \de  \Omega) + (S, \db \Omega) = (\overline S, \ddb \alpha) + (S, - \ddb \overline \alpha) = - (\ddb  \overline S, \alpha) + (\ddb S, \overline \alpha)= 0.$

Case (3): 
$(T, \Omega)=0$  because $(T, \Omega)=(T, \Psi)$ and $\Psi \in Ker d_{2p}$, $T \in Im d_{2p}' \subseteq (Ker d_{2p})^{\perp}$.

Case (4): 
$(T, \Omega)=0$  because $\Omega \in Ker \sigma_{2p}$ and $T \in Im \sigma_{2p}' \subseteq (Ker \sigma_{2p})^{\perp}$.

\bigskip

Let us prove now the converses (the technical details, which we shall prove all together in Proposition 4.6, are collected in the Claims). We refer to sequences (4.5) - (4.10).
\medskip

{\bf Case (1).} Let us denote by $P(M):= SP_p(M)$ the closed convex cone of strongly positive currents of bidimension $(p,p)$ (we choose this notation to emphasize that we could carry on the proof also with the cones $P_p(M)$ or $WP_p(M)$, and the corresponding dual cones of forms). 

Consider on $M$ a hermitian metric $h$ with associated $(1,1)-$form $\gamma$, and let $$\tilde P(M) := \{ T \in P(M) / (T, \gamma^p) = 1 \};$$
 it is a compact convex basis for $P(M)$ (in the sense of Sullivan \cite{Su}, Prop. I.5). 
 
 Our hypothesis can be written as: $Im \sigma_{2q+1}' \cap \tilde P(M) = \emptyset.$
\medskip

{\bf Claim (1).}  $Im \sigma_{2q+1}'$ is a closed linear subspace of ${\E}_{p,p}'(M)_{\R}$.
\medskip

Let us conclude the proof: by the Separation Theorem 4.5, there exists a closed hyperplane in ${\E}_{p,p}'(M)_{\R}$, strictly separating $Im \sigma_{2q+1}'$ and $\tilde P(M)$. Thus we get a $(p,p)-$form $\Omega$ such that 
$(T, \Omega)>0$ for all $T \in \tilde P(M)$ and 
$(T, \Omega)=0$ for all $T \in Im \sigma_{2q+1}'$.

This last condition means precisely that $\Omega \in (Im \sigma_{2q+1}')^{\perp} = Ker \sigma_{2q+1}$.

As for the first one, it assures that $\Omega$ is transverse. In fact, consider $T \in SP_p(M), T \neq 0$; then $(T, \gamma^p) = c > 0$, since $\gamma^p$ is a transverse form, and this implies that $c^{-1} T \in \tilde P(M)$, thus $(\Omega, c^{-1} T) > 0$ and also $(\Omega,  T) > 0$; this is sufficient by the Claim 
2.4.1.
\bigskip

{\bf Case (4)} is very similar, since now $Im \sigma_{2p}' \cap \tilde P(M) = \emptyset.$ To conclude as above, we have only to prove:
\medskip

{\bf Claim (4).}  $Im \sigma_{2p}'$ is a closed linear subspace of ${\E}_{p,p}'(M)_{\R}$.
\bigskip

{\bf Case (2).}  Our hypothesis is $(Im \sigma_{2q+1}' \cap Ker \sigma_{2p-1}') \cap \tilde P(M) = \emptyset.$
\medskip

{\bf Claim (2a).}  $Im \sigma_{2q+1}' \cap Ker \sigma_{2p-1}'$ is a closed linear subspace of ${\E}_{p,p}'(M)_{\R}$.
\medskip

By the Separation Theorem 4.5, there exists a closed hyperplane in ${\E}_{p,p}'(M)_{\R}$, strictly separating $Im \sigma_{2q+1}' \cap Ker \sigma_{2p-1}'$ and $\tilde P(M)$. Thus we get a $(p,p)-$form $\Omega$ such that 
$(T, \Omega)>0$ for all $T \in \tilde P(M)$ (that is, $\Omega$ is transverse) and 
$(T, \Omega)=0$ for all $T \in Im \sigma_{2q+1}' \cap Ker \sigma_{2p-1}'$, that is, 
$\Omega \in (Im \sigma_{2q+1}' \cap Ker \sigma_{2p-1}')^{\perp}$, which is the closure of the linear subspace $(Ker \sigma_{2q+1} + Im \sigma_{2p-1})$, and we conclude by the following Claim:
\medskip

{\bf Claim (2b).}  $Im \sigma_{2p-1}$ and $(Ker \sigma_{2q+1} + Im \sigma_{2p-1})$ are closed linear subspaces.
\medskip

{\bf Case (3).} Let us consider the Frechet space ${\E}_{2p}'(M)_{\R} = ( \oplus_{a+b=2p} \E_{a,b}'(M))_\R $, and denote by $\pi$ the projection on the addendum ${\E}_{p,p}'(M)_{\R}$. Moreover, consider
the closed convex cone $P'(M) = \{ R \in {\E}_{2p}'(M)_{\R} / R = \pi(R) , \pi(R) \in P(M) \}$, and its compact basis $\tilde P'(M) := \{ R \in P'(M) / (R, \gamma^p) = 1 \}$.
Our hypothesis is: $Im d_{2p}' \cap \tilde P'(M) = \emptyset.$
\medskip

{\bf Claim (3).}  $Im d_{2p}'$ is a closed linear subspace of ${\E}_{2p}'(M)_{\R}$.
\medskip

By the Separation Theorem 4.5, we get a real $2p-$form $\Psi = \sum_{a+b=2p} \Psi^{a,b}$, such that $d \Psi = 0$, since $\Psi \in (Im d_{2p}')^{\perp} = Ker d_{2p}$. Moreover, for every $R \in \tilde P'(M), \ (R,\Psi) = (R,  \Psi^{p,p}) > 0$; this assures as in case (1) that 
the 
$(p,p)-$form $\Omega := \Psi^{p,p}$ is transverse.
\medskip

We end the proof by the following remark.

\medskip
{\bf Remark 4.5.1.}
All the previous claims are proved, if we check that the following linear subspaces are closed: $Im \sigma_{2q+1}$, $Im \sigma_{2p-1}$, $Im \sigma_{2p}$, $Im d_{2p}$, 
and moreover that $Im \sigma_{2p-1}$ has finite codimension in $Ker \sigma_{2p}$ (this is done in Proposition 4.6).

 In fact notice that, by the Closed Range Theorem (see f.i. \cite{Sc} III.7.7), we can always switch from the dual operator to the operator itself (as regards the closure of the image). 

Moreover, let us consider $(Ker \sigma_{2q+1} + Im \sigma_{2p-1})$ in Claim (2b): 
both addenda are closed subspaces of $Ker \sigma_{2p}$.
Consider the quotient map $f : Ker \sigma_{2p} \to \frac{Ker \sigma_{2p}}{Ker \sigma_{2q+1}}$: it is a continuous surjective linear map, thus by Lemma 4.3, $f (Im \sigma_{2p-1})$ is closed in $\frac{Ker \sigma_{2p}}{Ker \sigma_{2q+1}}$, and so $$f^{-1} (f (Im \sigma_{2p-1})) = (Ker \sigma_{2q+1} + Im \sigma_{2p-1})$$ is closed in $Ker \sigma_{2p}$, and thus in $\E^{p,p} (M)_\R$.
\end{proof}

\medskip

\begin{prop} The following linear subspaces are closed: $Im \sigma_{2q+1}$, $Im \sigma_{2p-1}$, $Im \sigma_{2p}$, $Im d_{2p}$, 
and moreover $Im \sigma_{2p-1}$ has finite codimension in $Ker \sigma_{2p}$.
\end{prop}

\begin{proof} By Lemma 4.2, it is enough to prove  (for suitable spaces and maps)  that $ \frac{E}{Im f}$ is finite dimensional.
Now recall that, since $M$ is compact: $$dim H_{DR} ^{2p+1}(M, \R) = dim \frac{(Ker d_{2p+1})(M)}{(Im d_{2p})(M)} < \infty$$
$$dim H_{\ddb}^{p+1,p+1}(M, \R) = dim \frac{(Ker \sigma_{2p+1})(M)}{(Im \sigma_{2p})(M)} < \infty$$
$$dim H_{\de + \db}^{p,p}(M, \R) = dim \frac{(Ker \sigma_{2p})(M)}{(Im \sigma_{2p-1})(M)} < \infty.$$
As for the last assertions one may look at \cite{Bi1}, or at some other papers concerning Bott-Chern and Aeppli cohomology.

It remains to check $Im \sigma_{2q+1}$; as said before, we need: 
$dim \frac{(Ker \sigma_{2q+2})(M)}{(Im \sigma_{2q+1})(M)} < \infty.$

When $q > 0$, let us consider the sheaves involved in the sequence (4.2). Since $M$ is compact, it is well known that 
$H^k(M, \B^j) = 0$ for $k>0$, and $dim H^k(M, \L^j) < \infty$, since the sheaves $\Omega^{j}$ are coherent. Thus, for $k > 0$, $dim H^k(M, \H) < \infty$, using the cohomology sequence associated to  
 \begin{equation}\label{3}
 0 \to \R \stackrel{i}{\to}  {\O}   \stackrel{Re}{\to}  \H \to 0  
 \end{equation}
where $i(c) = ic , \  c \in {\R}$ , and  
$\ 2 Re f(z) = f(z) + \overline {f(z)}$.
\medskip

From the following short exact sequences arising from (4.2), 
 \begin{equation}\label{4}
 0  \to  {\H}  \to  {\L}^0 
\to {K}er \enskip \sigma_1  \to  0, \ \ \ \ \ 0  \to  {K}er \enskip \sigma_{1}  \to  {\L}^{1}
  \to  {K}er \enskip \sigma_{2}  \to 0,  
\end{equation}
 $$0  \to  {K}er \enskip \sigma_{2}  \to  {\L}^{2}
  \to  {K}er \enskip \sigma_{3}  \to 0, \ \dots \ 0  \to  {K}er \enskip \sigma_{2q-1}  \to  {\B}^{2q-1}
  \to  {K}er \enskip \sigma_{2q}  \to 0,$$  
$$0 \to {K}er \enskip \sigma_{2q}  \to  {\E}^{q,q}_{\R}
 \to {K}er \enskip \sigma_{2q+1} \to  0, \ \ \ 
 0 \to {K}er \enskip \sigma_{2q+1}  \to  {\E}^{q+1,q+1}_{\R}
 \to {K}er \enskip \sigma_{2q+2} \to  0$$

\medskip
 we get: by the first one, $dim H^k(M, Ker \sigma_{1}) < \infty$ for $k>0$, which implies,
by the second one, $dim H^k(M, Ker \sigma_{2}) < \infty$ for  $k>0$, and so on. 
And finally 
$dim \frac{(Ker \sigma_{2q+2})(M)}{(Im \sigma_{2q+1})(M)} < \infty$
since $ dim H^1(M, Ker \sigma_{2q+1})< \infty.$
\medskip

When $q=0$, use the sequence (4.1) to get the same result.
\end{proof}
\medskip

Thus we ended the proof of Theorem 3.2.

\bigskip

\section{Exact generalized \pkk forms}

Let us begin with an example. It is well known that a 1K form on a compact K\"ahler manifold cannot be exact, because on the contrary we would have:
$$0 < vol(M) = \int_M \omega^n = \int_M d\alpha \wedge \omega^{n-1} = \int_M d(\alpha \wedge \omega^{n-1}) = 0.$$
But this is no more true when $p >1$. We recall here an example proposed by Yachou \cite{Y}, which illustrates the following result (see \cite{GR}, pp. 506-507): If $G$ is a complex connected semisimple Lie group, it has a discrete subgroup $\Gamma$ such that the homogeneous manifold $G/\Gamma$ is compact, holomorphically parallelizable and does not have hypersurfaces (since $a(M) = 0$).

{\bf Example 5.1} Take $G=SL(2,\C)$, $\Gamma = SL(2,\Z)$, and  consider the holomorphic $1-$forms $\eta, \alpha, \beta$ on $M := G/ \Gamma$ induced by the standard basis for $\g^*$: it holds
$$ d \alpha = -2 \eta \wedge \alpha, \ \ d \beta = 2 \eta \wedge \beta, \ \ d \eta = \alpha \wedge \beta.$$
The standard fundamental  form, given by $\omega = \frac{i}{2} (\alpha \wedge \overline \alpha + \beta \wedge \overline \beta +\eta \wedge \overline \eta )$, satisfies $d \omega^2=0$, so that $\omega^2$ is a balanced form: but it is exact, since
$$\omega^2 = d(\frac{1}{16} \alpha \wedge d \overline  \alpha + \frac{1}{16} \beta \wedge d \overline  \beta +\frac{1}{4} \eta \wedge d \overline \eta ).$$
Hence this manifold does not support not only hypersurfaces, but also closed positive $(1,1)-$currents.

\bigskip
As a matter of fact, Sullivan considered also exact forms in \cite{Su}, Theorem I.7 (see in the Introduction the second part of the cited result of Sullivan), but this argument was not developed further  by Harvey and Lawson, since on compact manifolds no K\"ahler form can be exact.

We just showed that when $p > 1$ the situation is very different, hence we shall study the general case in the following Theorem, concerning 
 \lq\lq exact\rq\rq generalized $p-$K\"ahler  forms. Some remarks on this theorem are collected after the proof.
\medskip

\begin{thm}
Let $M$ be a {\rm compact} complex manifold of dimension $n \geq 2$, and let $p$ be an integer, $1 \leq p \leq n-1$; denote $q := p-1$ in the subscript of the operators in (4.5) - (4.10). Then:

\begin{enumerate}

\item There is a  transverse $(p,p)-$form $\Omega$ on $M$ such that $\Omega \in Im \sigma_{2q} \ \iff $
there are no non trivial currents $T \in {\E}_{p,p}'(M)_{\R}$, $T \geq 0$, $T \in Ker \sigma_{2q}'.$ 

\item There is a transverse $(p,p)-$form $\Omega$ on $M$ such that $\Omega \in Im d_{2p-1} \ \iff $
there are no real currents $R \in {\E}_{2p}'(M)_{\R}$, $R_{p,p} := T \geq 0$, $T \neq 0$, $R \in Ker d_{2p-1}'.$  

\item There is a transverse $(p,p)-$form $\Omega$ on $M$ such that $\Omega \in Im \sigma_{2p-1} \cap Ker \sigma_{2q+1}$  $\iff $
there are no non trivial currents $T \in {\E}_{p,p}'(M)_{\R}$, $T \geq 0$, $T \in (Im \sigma_{2q+1}' + Ker \sigma_{2p-1}').$ 

\item there is a transverse $(p,p)-$form $\Omega$ on $M$ such that $\Omega \in Im \sigma_{2p-1} \ \iff $
there are no non trivial currents $T \in {\E}_{p,p}'(M)_{\R}$, $T \geq 0$, $T \in Ker \sigma_{2p-1}'.$ 

\end{enumerate}
\end{thm}

\begin{proof} As in Theorem 3.2, one side is straightforward. Also the other side of 
the proof is similar to that of Theorem (3.2)': in the present case, we shall separate positive currents from \lq\lq closed\rq\rq currents, so that the separating hyperplane turns out to be a transverse \lq\lq exact\rq\rq form. What we need is a result similar to Proposition 4.6. 
\medskip

Let us give the details. As for the case (4), first of all notice that we require a transverse $(p,p)-$form $\Omega$ on $M$ such that $\Omega = \db \beta + \de \overline \beta$, or, which is the same, a $2p-$form $\Psi = d \Gamma$ such that $\Psi^{p,p} := \Omega > 0$. Thus $\Psi$, but not $\Omega$, is exact in the classical sense.

Let us denote by $P(M):= SP_p(M)$ the closed convex cone of strongly positive currents of bidimension $(p,p)$.
Consider on $M$ a hermitian metric $h$ with associated $(1,1)-$form $\gamma$, and let $\tilde P(M) := \{ T \in P(M) / (T, \gamma^p) = 1 \}$:
 it is a compact convex basis for $P(M)$. 
 
 Our hypothesis can be written as: $Ker \sigma_{2p-1}' \cap \tilde P(M) = \emptyset.$
By the Separation Theorem 4.5, there exists a closed hyperplane in ${\E}_{p,p}'(M)_{\R}$, strictly separating $Ker \sigma_{2p-1}' $ and $\tilde P(M)$. Thus we get a $(p,p)-$form $\Omega$ such that 
$(T, \Omega)>0$ for all $T \in \tilde P(M)$ and 
$(T, \Omega)=0$ for all $T \in Ker \sigma_{2p-1}' $.
This last condition means precisely that $\Omega \in (Ker \sigma_{2p-1}' )^{\perp}$ which is the closure of  $Im \sigma_{2p-1}$.

As for the first condition, it assures that $\Omega$ is transverse. In fact, consider $T \in SP_p(M), T \neq 0$; then $(T, \gamma^p) = c > 0$, since $\gamma^p$ is a transverse form, and this implies that $c^{-1} T \in \tilde P(M)$, thus $(\Omega, c^{-1} T) > 0$ and also $(\Omega,  T) > 0$; this is sufficient by the Claim 
2.4.1.
Thus it remains to prove:

{\bf Claim (4).}  $Im \sigma_{2p-1}$ is a closed linear subspace of ${\E}^{p,p}(M)_{\R}$.
\medskip

\medskip

Case (1) is very similar, since now $Ker \sigma_{2q}' \cap \tilde P(M) = \emptyset.$ To conclude as above, we have only to prove:

{\bf Claim (1).}  $Im \sigma_{2q}$ is a closed linear subspace of ${\E}^{p,p}(M)_{\R}$.
\bigskip

Case (3). Notice that the hypothesis on $\Omega$ is equivalent to $\Omega = \db \beta + \de \overline \beta$, with $\ddb \beta =0$, while the condition on $T$ assures that $\de T = \ddb S$.

Therefore  we start from $(Im \sigma_{2q+1}' + Ker \sigma_{2p-1}') \cap \tilde P(M) = \emptyset.$
\medskip

{\bf Claim (3a).}  $Im \sigma_{2q+1}' + Ker \sigma_{2p-1}'$ is a closed linear subspace of ${\E}_{p,p}'(M)_{\R}$.
\medskip

By the Separation Theorem 4.5, there exists a closed hyperplane in ${\E}_{p,p}'(M)_{\R}$, strictly separating $Im \sigma_{2q+1}' + Ker \sigma_{2p-1}'$ and $\tilde P(M)$. Thus we get a $(p,p)-$form $\Omega$ such that 
$(T, \Omega)>0$ for all $T \in \tilde P(M)$ (that is, $\Omega$ is transverse) and 
$(T, \Omega)=0$ for all $T \in Im \sigma_{2q+1}' + Ker \sigma_{2p-1}'$, that is, 
$\Omega \in (Im \sigma_{2q+1}' + Ker \sigma_{2p-1}')^{\perp}$, which is given by $Ker \sigma_{2q+1}$ intersected the closure of the linear subspace  $Im \sigma_{2p-1}$, and we conclude by the following Claim:
\medskip

{\bf Claim (3b).}  $Im \sigma_{2p-1}, Im \sigma_{2q+1}'$  are closed linear subspaces.
\bigskip

Case (2). Let us consider the Frechet space ${\E}^{2p}(M)_{\R} = ( \oplus_{a+b=2p} \E^{a,b}(M))_\R $, and denote by $\pi$ the projection on the addendum ${\E}^{p,p}(M)_{\R}$. Moreover, consider
the set $$A := \{\Psi \in {\E}^{2p}(M)_{\R} / 
 \exists c > 0 \ {\rm such \  that} \ \pi(\Psi) > c \gamma^p \}$$
(the condition obviously means that $\pi(\Psi) - c \gamma^p$ is strictly weakly positive).

It is easy to control that $A$ is a non-empty convex open subset in the topological vector space ${\E}^{2p}(X)_{\R}$. If there is no form $\Omega$ as stated in the Theorem, then we get also $$(A \cap {\E}^{p,p}(X)_{\R} ) \cap Im d_{2p-1} = \emptyset.$$

{\bf Claim (2).}  $Im d_{2p-1}$ is a closed linear subspace of ${\E}^{2p}(M)_{\R}$.
\medskip

By the Hahn-Banach Theorem 4.4, we get a separating closed hyperplane, which is nothing but a current $R \in {\E}_{2p}'(X)_{\R}$, for which we can suppose:

$R \in (Im d_{2p-1})^{\perp} = Ker d_{2p-1}'$,

$(R, \Omega) = (R_{p,p}, \Omega)>0$ for every $\Omega \in (A \cap {\E}^{p,p}(X)_{\R})$ (thus $T:= R_{p,p} \neq 0$).

Let us check that $T \geq 0$, i.e. $(T, \Omega) \geq 0, \ \forall \Omega \in WP^{p,p}(X)$. For every $\varepsilon >0$, $\Omega + \varepsilon \gamma^p \in A$, thus 
$$(T, \Omega) = (T, \lim_{\varepsilon \to 0} \Omega + \varepsilon \gamma^p) = \lim_{\varepsilon \to 0} (T, \Omega + \varepsilon \gamma^p) \geq 0.$$

We end the proof by using the  forthcoming Proposition 5.2.
\end{proof}
\medskip

\begin{prop} The following linear subspaces are closed:  $Im \sigma_{2q+1}$, $Im \sigma_{2p-1}$, $Im \sigma_{2q}$, $Im d_{2p-1}$, 
and moreover  $Im \sigma_{2q+1}'$ has finite codimension in $Ker \sigma_{2q}'$.
\end{prop}

\begin{proof}  See the proof of Proposition 4.6. Notice moreover that, for $p+k=n$, 
$$H_{\de + \db}^{k,k}(M, \R) \simeq \frac{\{T
\in {\E '}_{p,p}(X)_{\R}; i\ddb T =0\}}{\{ \de \overline S + \db S ; S \in  {\E '}_{p,p+1}(X)
\}} =  \frac{(Ker \sigma_{2q}')(M)}{(Im \sigma_{2q+1}')(M)}$$
since the cohomology groups can be described using either forms or currents of the same bidegree. 
\end{proof}

\bigskip

{\bf 5.2.1 Remark.} Notice that, as before,
$5.1(1)  \Longrightarrow  5.1(2)  \Longrightarrow  5.1(3)  \Longrightarrow  5.1(4)$, and moreover, for every $j$, $5.1(j)  \Longrightarrow  3.2(j)$.
The stronger condition, 5.1(1), is in fact a $p-$K\"ahler condition with exact (that means $\ddb-$exact) form.
\medskip

{\bf 5.2.2 Remark.} The statement of Theorem I.7  in \cite{Su}  is the following:
\lq\lq  If no non-trivial structure cycle exists, some transversal closed form is cohomologous to zero\rq\rq. Of course, also the converse holds.

It can be translated in our situation (where $M$ is a compact complex manifold) as follows: 
\lq\lq $M$ has a real $2p-$form $\Psi = \sum_{a+b=2p} \Psi^{a,b}$, such that $ \Psi = d \Gamma$ and the 
$(p,p)-$form $\Omega := \Psi^{p,p}$ is 
transversal,   if and only if $M$ has no strongly  positive currents $T \neq 0$, of bidimension $(p,p)$, such that $dT = 0$\rq\rq.

The condition on the form $\Omega$ is equivalent to say that there is a real $(p,p)-$form $\Omega >0$ with $\Omega= \de  \overline \beta + \db \beta$
for some $(p,p-1)-$ form $\beta$: this means $\Omega \in Im \sigma_{2p-1}$, while $d T = 0$ is equivalent to the condition 
 $T \in Ker \sigma_{2p-1}'.$ Thus the statement of Sullivan is  5.1(4).
\medskip

{\bf 5.2.3 Remark.} 
Sullivan noticed also in III.10 that for $p=1$, there are always non trivial structure cycles, so that on a compact manifold there is no hermitian metric $h$ whose K\"ahler form is the $(1,1)-$component of a boundary in the sense of 5.1(4).

If we look at the whole Theorem 5.1, we can prove in fact that when $p=1$, it reduces to the existence of \lq\lq closed\rq\rq positive currents, since \lq\lq exact\rq\rq transverse forms never exist, due to the compactness of $M$. This is obvious for the statement 5.1(1), since we would have $\omega = i \ddb f >0$, a non-constant plurisubharmonic function on a compact manifold.

But in general, if $\omega \in Im \sigma_{2p-1}$, then $\omega := \psi^{1,1}$, the $(1,1)-$component of an exact form $\psi = d \gamma$. Thus it would give: 
$0 = \int_M (d \gamma)^n = \int_M \psi^n = \int_M \omega^n >0.$

When $p > 1$, the situation changes, as seen in Example 5.1; there, it is easy to check that $\omega^2 \in Im \sigma_{2q}$, i.e. $\omega^2 = i \ddb \gamma$, so that all conditions in Theorem 5.1 make sense, for $p=2$.
\medskip

{\bf 5.2.4 Remark.} It is also interesting to notice that the above conditions on forms can be described as: \lq\lq The null class in cohomology contains a transverse form\rq\rq, where the cohomology groups are: $H_{\ddb}^{p,p}(M, \R)$ for 5.1(1),
$H_{d}^{p,p}(M, \R)$ for 5.1(2),
$H_{\de + \db}^{p,p}(M, \R)$ for 5.1(4). For 5.1(3) the class is that of $g(\Omega)$, where $g$ is the map induced by the identity:
$g: H_{d}^{p,p}(M, \R) \to H_{\de + \db}^{p,p}(M, \R)$.

\bigskip

\section{The non-compact case}

While a 2K form can be exact on a compact manifold (as we have seen in the previous section), the natural environment of \lq\lq exact\rq\rq generalized  \pkk forms is that of non-compact manifolds; indeed, $\C^n$ and Stein manifolds are K\"ahler with a form $\omega = i \ddb u$ ($u$ is a smooth strictly plurisubharmonic function).

Other classes of non-compact manifolds where one could look for generalized \pkk structures are $q-$complete and $q-$convex manifolds. Let us recall here the definitions, which are not uniform in the literature (see also \cite{De}, IX.(2.7) for analytic schemes).

\begin{defn} A manifold $X$ of complex dimension $n$ is said to be {\it strongly $q-$convex} (for brevity, {\it $q-$convex}) if it has a smooth exhaustion function $\psi
: X \to \R$ which is strongly $q-$convex outside an exceptional compact set $K \subset X$ (this means that $X - K$ has an atlas such that, in local coordinates, $(i \ddb \psi)(x)$ has at least $(n-q+1)$ positive eigenvalues, for all $x \in X$).

We say that $X$ is {\it $q-$complete} if $\psi$ can be chosen so that $K = \emptyset$.
\end{defn}

Thus $q=1$ is the strongest property, and  $1-$complete manifolds corresponds to Stein manifolds, since the strongly $1-$convex functions are just the strictly plurisubharmonic functions.

By convention, a {\it compact} manifold $M$ is said to be (strongly) $0-$convex (with $K=M$).
\medskip

{\bf 6.1.1 Remark.} Let $\F \in Coh(X)$, i.e. let $\F$ be a coherent sheaf on $X$. Then (see f.i. \cite{De}, IX.4 and \cite{AG} n. 20):
\begin{enumerate}
\item If $X$ is compact, then $dim H^j(X,\F) < \infty\ \forall j$; 
\item If $X$ is $q-$convex, then $dim H^j(X,\F) < \infty$ when $j \geq q$; 
\item If $X$ is $q-$complete, then $H^j(X,\F) =0$ when $j \geq q$. 
\end{enumerate}

In fact, the $1-$convex spaces can be characterized by that property, and also by the existence of the Remmert reduction, as the following theorem shows
(\cite{CM}):

{\it Theorem.}  
The following statements are equivalent, for a complex analytic space $X$:
\begin{enumerate}
\item $X$ is $1-$convex;
\item For every $\F \in Coh(X)$, $dim H^j(X,\F) < \infty$ when $j \geq 1$;

\item $X$ is obtained from a Stein space by blowing up finitely many points (this is called the Remmert reduction).
\end{enumerate}

\bigskip

As for $q-$complete manifolds, Barlet proved in \cite{Bar} (Proposition 3) the following fact:

\begin{prop} Let $X$ be a complex manifold, and let $\psi:X \to \R^+$ be a smooth proper function which is strongly $q-$convex at each point of $X$. Then there is a transverse $(q,q)-$form $\Omega$ such that $\Omega = i \ddb \theta$ for some real $(q-1,q-1)-$form $\theta$.
\end{prop}

This gives an interesting result:

\begin{cor} A $q-$complete manifold $X$ is $p-$K\"ahler for ever $p \geq q$, with a $\ddb-$exact form.
\end{cor}

Hence we have here a remarkable class of balanced manifolds (with a $\ddb-$exact form; moreover, it is not hard to verify that $\Omega$ is in fact strictly positive): that of $q-$complete manifolds ($q < n$).

\medskip

As regards $q-$convex manifolds, a similar  result does not hold, in general, also when $q=1$. Classical results due to Coltoiu \cite{Co} asserts that a $1-$convex manifold $X$, with an irreducible curve $S$ as exceptional set, is K\"ahler when $dim X \neq 3$ or when $S$ is not a rational curve. Moreover, 
we characterized in \cite{AB7}, Theorem I, precisely those $1-$convex threefolds, with $1-$dimensional exceptional set, that admit a K\"ahler metric, that is: \lq\lq When $S$ is an irreducible curve, then $X$ is K\"ahler if and only if the fundamental class of $S$ does not vanish in $H^{2n-2}_c (X)$\rq\rq .

When $dim S > 1$, we got some results in \cite{ABL} as regards $p-$K\"ahler structures, as we shall explain now.
A $1-$convex manifold $X$ of dimension $n$ is given by a desingularization $f:X \to Y$ of a Stein space $Y$ which has just a finite number of (isolated) singularities ($f$ is the Remmert reduction, see Remark 6.1.1); the exceptional set $S$, which is $f^{-1}(Sing Y)$, has dimension $k \leq n-2$ and is the maximal compact analytic subset of $X$. Obviously, $X-S$ carries an exact 1K form $\omega_S = i \ddb f^*u$ coming from $Y -$ Sing $Y$; but when $n \geq 3$, in general we cannot extend $\omega_S$ to a closed transverse form $\omega$ across $S$: indeed, there are quite simple examples of non-K\"ahler $1-$convex threefolds.

Nevertheless, we got, among others:

\begin{thm} {\rm (see Theorem 4.2, Proposition 4.4 and Theorem 4.12 in \cite{ABL})} 
\begin{enumerate}

\item Let $X$ be a complex $n-$dimensional manifold, let $S$ be an exceptional subvariety  of $X$, such that $X - S$ has a $\ddb$-exact K\"ahler form. Then $X$ is $p-$K\"ahler for every $p > dim S$, with a $\ddb$-exact $p-$K\"ahler form.

\item Let $X$ be a 1-convex manifold with exceptional set $S$ of dimension $k$. Then $X$ is $p-$K\"ahler for every $p > k$, with a $\ddb$-exact $p-$K\"ahler form; in particular, a 1-convex manifold is always balanced (with a $\ddb$-exact form). Moreover, if $k > \frac{n-1}{2}$, then $X$ is also $k-$K\"ahler.

\end{enumerate}
\end{thm}

The proof of these assertions is based on classical separation's results  between forms and compactly supported currents: this is the tool that we would like to develop in what follows, looking at generalized \pkk structures: maybe this kind of use of the duality could be interesting from his own.
\bigskip

Let $X$ be a complex manifold of dimension $n$. A positive (analytic) $q-$cycle of $X$ is a finite linear combination of irreducible $q-$dimensional compact subvarieties of $X$, with positive integers as coefficients; we shall write:
$Y = \sum n_j Y_j \in C^+_q(X).$

The spaces $C^+_q(X) \subseteq SP_q(X)$ have been intensively studied in the period 1960-70 (by Andreotti, Norguet, Barlet and others), and gave a motivation to the study of positive (closed) currents with compact support on non-compact manifolds.
For example, in \cite{Bar} Barlet proved that: \lq\lq If $X$ is a $q-$complete analytic space, then $C^+_{q-1}(X)$ is a Stein space\rq\rq .

As we noted above, when $X$ is compact, all relevant cohomology groups are finite dimensional, while this is not the case in general; some remarkable cases (Stein, $q-$complete, $q-$convex) have been studied in the sixties and seventies.

We can go on following two ways: one way is to  assure finite dimensionality of the right cohomology groups, and then proceed as in Proposition 4.6. In this setting, let us recall the following result (Propositions (5.3), (5.4), but see also (5.3)', (5.4)', (5.5), (5.5)' in \cite{NoSi}, which correct a wrong statement in \cite{Bi1}):

\begin{prop} Let $X$ be a strongly $q-$convex manifold of dimension $n$, let $s > q$; then:
\begin{enumerate}

\item If $dim H_{DR}^{2s+1}(X, \C) < \infty$, then $dim H_{\de + \db}^{s,s}(X) < \infty$.

\item If $dim H_{DR}^{2s}(X, \C) < \infty$, then $dim H_{\ddb}^{s,s}(X) < \infty$.
\end{enumerate}
\end{prop}
\medskip

But it seems more interesting  to notice that in Proposition 4.6 what we actually need is the closeness of some subspaces, to use separation's theorems
in the proof of Theorem 3.2; and this condition is equivalent to ask that the involved cohomology groups are Hausdorff (i.e. Fr\'echet), or to ask that some operators ($d, \de + \db, \ddb$) are topological homomorphisms, which is the same.

Serre pointed out (\cite{Se}, pp. 22-23) that the behavior of the operator $d = d_s : {\E}^{s}(X) \to {\E}^{s+1}(X)$ is very different from that of the operator $\db$: while $d$ is always a topological homomorphism, this is not the case for $\db$ (see the simple example given ibidem, n. 14). A sufficient condition is given in \cite{Se}, Proposition 6:
\begin{prop}  If $dim H^q(X, \Omega^p) < \infty$, then $\db : \E^{p,q-1}(X) \to \E^{p,q}(X)$ is a topological homomorphism.
\end{prop}

Notice that the proof of this result  is straightforward: the statement $$dim H^{p,q}_{\db} (X) < \infty$$ is equivalent to say that the space of boundaries with respect to the operator $\db$ has finite codimension in the space of cycles, which is closed. Hence $\db$ is a topological homomorphism by Lemma 4.2.

Our situation is much more complicated: nevertheless we can get the required topological homomorphisms more or less in the same hypotheses of Proposition 6.6, as we shall prove in the next section.
Since  $dim H^j(X,\Omega^k) < \infty$ when $j \geq q$ for a $q-$convex manifold $X$, as said in Remark 6.1.1, let us go on along this way.

\bigskip

\section{$H_{\de + \db}^{k,k}(X, \R)$ and $H_{\ddb}^{k,k}(X, \R)$  are Hausdorff}

It is well known that, for every complex manifold $X$,
the De Rham cohomology groups $H_{DR}^j (X, \R)$ are Hausdorff topological vector spaces, i.e. the differential operator $d$ is a topological homomorphism; but this is not the case, for instance, for the operator $\db$, as we said above.
Thus we shall study from this point of view the operators that appear in the sequence (4.2). 

\medskip

Let us recall the notation and consider some preliminary results on topological vector spaces; we refer to \cite{Ca1} and \cite{Ca2}. 

\begin{defn} {\rm  (see \cite{Ca1}, page 311)} A sheaf ${\F}$ 
of t.v.s. on a complex manifold $X$
is called a {\it Fr\'echet sheaf} if for every open subset $U$ of $X$, ${\F} (U)$ is a
Fr\'echet space, and for every open subset  $V$ of $X$ such that $U \subseteq V$, then the map $\rho^V_U : 
{\F} (V) \to {\F} (U)$ is continuous.

If ${\F}$ and ${\G}$ are Fr\'echet sheaves, and $\sigma : {\F} \to 
{\G}$ is a sheaf homomorphism, $\sigma$ is called a {\it Fr\'echet homomorphism}
if for every open subset $U$ of $X$, $\sigma (U) : {\F} (U) \to 
{\G} (U)$ is continuous.
\end{defn}

\medskip
{\bf 7.1.1 Remark.} The sequences (4.1) and (4.2) are exact sequences of Fr\'echet sheaves and 
Fr\'echet homomorphisms (recall  that ${\H} (U) = Ker\  i \partial \overline \partial :
{\E}^{0,0}_{\bf R} (U) \to {\E}^{1,1}_{\bf R} (U) $).

\bigskip

If ${\F}$ is a Fr\'echet sheaf on $X$ and ${\U}$ is a countable
covering for $X$, for every $q \geq 0$ we put on $C^q ({\U} , {\F})$ the product topology,
which is Fr\'echet. The maps $\delta^q : C^q ({\U} , {\F}) \to
C^{q+1} ({\U} , {\F})$ becomes continuous.
Moreover, 
$$H^q ({\U} , {\F}) =  {{Z^q ({\U} , {\F})} \over 
{B^q ({\U} , {\F})}}  =  {{Ker \delta^q} \over {Im \delta^{q-1}}}$$ is endowed with the quotient topology (which is Fr\'echet if and only if it is 
Hausdorff), and
$$H^q (X,{\F}) = \lim_{\to}  H^q ({\U} , {\F})$$ 
is endowed with the direct limit topology.

\begin{defn} A Fr\'echet sheaf ${\F}$ is {\it normal} if there is a Leray covering ${\U}$ of $X$ for ${\F}$ such that, for every covering
${\V}$ of $X$, there is a covering ${\W} \subset {\U}$ of $X$ which is a refinement of ${\V}$.
\end{defn}

\begin{prop} There exists a covering ${\A}$ of $X$ which is a Leray covering for all sheaves
involved in the sequences (4.11) and (4.12). Moreover, all these sheaves are normal with respect  to ${\A}$.
\end{prop}

\begin{proof} We adapt a construction given in \cite{Ca2}.
 Fix a riemannian metric on $X$, and denote by $B(x,a)$ 
the geodesic ball of center $x$ and radius $a$.
 Notice that, given a (holomorphic) chart $(U,\varphi)$ of $X$, 
it is possible to choose $U$ such that every geodesic ball $B(x,a) \subset U$
has a convex image in 
$\varphi (U)$; in this case, $B(x,a)$ is an open Stein subset of $U$,
because its image in $\varphi (U)$ is holomorphically convex.

Choose a locally finite open covering of $X$, 
${\U} = \{ U_i , \varphi_i \}_{i \in I}$, with the above property; for every
$x \in X$, call
$$r(x) = sup \{ a / B(x,a) \subset U_i \  \rm{for \  some}\  i \in I \}.$$ 
Then choose an exaustion sequence $\{ K_j \}_{j \geq 1}$ of compact subsets of $X$ and
a decreasing sequence of real numbers $\{ c_j \}_{j \geq 1}$ such that 
$0 < c_j < d(K_j , X - (K_{j+1})^{int})$.

For every $x \in X$, let us denote by $l(x)$ the index such that
$x \in K_{l(x)} - K_{l(x) - 1}$, and let us fix $a(x) > 0$ such that
$$B(x,a(x)) \subset \bigcup_{x \in U_i} U_i \quad , \quad a(x) < min \  \{ 
{1\over 3}\  c_{l(x)} , {1\over 3}\  \min_{K_{l(x) + 2}} r \} .$$

Let ${\A} = \{ B(x,a) \}_{x \in X , \ a < a(x)}$: we shall check that every finite 
intersection of elements of ${\A}$ is Stein and contractible. First of all,
every element of ${\A}$ satisfies the request, so that also the finite 
intersections of elements of ${\A}$ are Stein.

Take 
$$B = B(x_0 , a_0 ) \cap \dots \cap B(x_p , a_p ) \not= \emptyset \quad , \quad p \geq 1 ;$$
it is enough to show that $B(x_0 , a_0 ) \cup \dots \cup B(x_p , a_p )$
is contained in a fixed chart $(U_0, \varphi_0)$ of the covering ${\U}$,
because in that case every ball becomes, via the biholomorphic map $\varphi_0$,
a convex set in $\varphi_0 (U_0)$, so that also the intersection of the image of the balls is convex there, and
also contractible.

It is easy to check that, since $B \not= \emptyset$, for every $h , k \in \{0 , \dots , p \}$, 
the intergers $l_k := l(x_k)$ and $l_h := l(x_h)$ satisfy $\vert l_k - l_h \vert \leq 1$.

Moreover, for every $ k \in \{1 , \dots , p \} \ , \ \forall x \in B(x_k , a_k )$ it holds

$$dist \  (x,x_0) < 2a_k + a_0 < {2\over 3} \min_{K_{l_k + 2}} r  + 
{1\over 3}\ \min_{K_{l_0 + 2}} r \leq \min_{K_{l_0 }} r  \leq r(x_0) ,$$

\noindent since $x_0 \in K_{l_0}$. So for all $k , B(x_k , a_k ) \subseteq B(x_0 , r(x_0) ) \subset U_i$.

We proved that, if $V$ is a finite intersection of elements of ${\A}$,
 $V$ is Stein and contractible; hence for all $j > 0$, 
$H^j (V,{\R}) = 0$ and $H^j (V,{\G}) = 0 \ \forall \  {\G} \in
Coh (X)$. From the sequence (4.11), also $H^j (V,{\H}) = 0$ for all $j > 0$, which implies
 that in (4.12)  $H^j (V,{K}er \enskip \sigma_1) = 0$ for all $j > 0$, and so on, as in the proof of Proposition 4.6.
\end{proof}

\begin{prop} {\rm (see \cite{Ca1} pages 312-313)} Let
$$0  \to {\F} '  \to {\F} \to  {\F} ''  \to 0$$
\noindent be an exact sequence of Fr\'echet sheaves and Fr\'echet
homomorphisms. If ${\U}$ is a countable Leray covering of $X$ for
${\F} '$, then  the maps
$\delta_*^q : H^q ({\U} , {\F} '') \to
H^{q+1} ({\U} , {\F} ')$ are continuous. If moreover $H^{q+1} ({\U} ,
 {\F})$ is Hausdorff, then $\delta_*^q$ is a topological homomorphism.

Let ${\G}$ be a normal Fr\'echet sheaf and ${\V}$ a countable Leray
covering of $X$ for ${\G}$; then $ H^q ({\V} , {\G}) \to
H^q (X , {\G})$ \  is a topological isomorphism for every $q$.
\end{prop}

\begin{cor} Let
$$0 \to  {\F} '  \to  {\F} \to  {\F} ''  \to  0$$
\noindent be an exact sequence of normal Fr\'echet sheaves 
(with respect to  a countable
Leray covering ${\U}$ of $X$) and Fr\'echet
homomorphisms. If  $H^{q} ({\U} , {\F}) = H^{q+1} ({\U} , {\F}) = 0$,
then the coboundary map 
$\ \delta^q_* : H^{q} (X, {\F} '') \to H^{q+1} (X, {\F} ') \ $
is a topological isomorphism.
\end{cor}

\begin{prop} Let
$$0 \to {\F} '  \to  {\F} \stackrel{\sigma}{\to}  {\F} ''  \to  0$$
\noindent be an exact sequence of normal Fr\'echet sheaves 
(with respect to a countable Leray covering $\U$ of $X$), and Fr\'echet
homomorphisms.

(i) If $dim \  H^q (X,{\F}) < \infty$, then $H^q (X,{\F})$ is Hausdorff.

(ii) If $dim \  H^q (X,{\F}) < \infty$ and  $ H^{q+1} (X,{\F} ')$ is
Hausdorff, then  $ H^{q} (X,{\F} '')$ is Hausdorff.
\end{prop}

\begin{proof} By Proposition 7.4, there is a topological isomorphism
$ H^q ({\U} , {\F}) \to H^q (X , {\F})$ for every $q$ , hence we can argue on 
$H^q ({\U} , {\F}) =$ $ {Z^q ({\U} , {\F})} \over 
{B^q ({\U} , {\F})} $.

If $q=0$, $H^0 ({\U} , {\F}) = {Z^0 ({\U} , {\F})} = Ker  \delta^0$ is Hausdorff; if $q > 0$, we get (i) by Lemma 4.2 
since the map $\delta^{q-1} : C^{q-1} ({\U} , {\F}) \to
Z^{q} ({\U} , {\F}) \ $ is continuous between Fr\'echet spaces.
\smallskip

To prove (ii), consider the following diagram (for $q > 0$):
\smallskip

$C^{q-1} ({\U} , {\F})  \rightarrow  C^{q-1} ({\U} , {\F} '')  \to 0$

$\ \ \ \ \ \ \ \downarrow \ \ \ \ \ \  \ \ \ \ \ \ \ \downarrow$

$Z^{q} ({\U} , {\F}) \stackrel{\sigma_q}{\to}  Z^{q} ({\U} , {\F} '')$
 
\smallskip

In this diagram, $\sigma_q (B^{q} ({\U} , {\F})) =
B^{q} ({\U} , {\F} '')$  and  $\sigma_q (Z^{q} ({\U} , {\F}))
 = Ker \  (\delta^q_* \circ \pi)$, 
where $\pi : Z^{q} ({\U} , {\F} '') \to H^{q} ({\U} , {\F} '')$ and
$\delta^q_* : H^{q} ({\U} , {\F} '') \to H^{q+1} ({\U} , {\F} ')$.

By  part (i), $B^{q} ({\U} , {\F})$ is closed and has
finite codimension in $Z^{q} ({\U} , {\F})$; moreover, 
$Ker \  (\delta^q_* \circ \pi)$ is a closed subspace of $Z^{q} ({\U} , {\F} '')$
 (because $H^{q+1} ({\U} , {\F} ')$ is Hausdorff), hence it is Hausdorff. Thus
we can apply Lemma 4.3 to $\sigma_q : Z^{q} ({\U} , {\F}) \to 
Ker \  (\delta^q_* \circ \pi)$, with $N = B^{q} ({\U} , {\F})$; this
implies that $\sigma_q (B^{q} ({\U} , {\F} )) = B^{q} ({\U} , {\F} '')$
is closed, hence $H^{q} ({\U} , {\F} '')$ is Hausdorff.
\end{proof}

\medskip
Let us use now these results to get some information on cohomology; a particular case of the following theorem is Corollary 2.5 in \cite{ABL} (compare also with Proposition 6.6).

\begin{thm}  Let $X$ be a complex manifold, let $q \geq 0$.
\begin{enumerate} 
\item If $dim \ H^j (X, \Omega^{2q+1-j} ) < \infty \quad \forall j \in \{q+1, \dots, 2q+1 \}$, then 
$H_{\ddb}^{q+1,q+1}(X, \R)$ is Hausdorff.

\item If $dim \ H^j (X, \Omega^{2(q+1)-j} ) < \infty \quad \forall j \in \{q+1, \dots, 2(q+1) \}$, then $H_{\de + \db}^{q+1,q+1}(X, \R)$ is Hausdorff.
\end{enumerate}
\end{thm}

\begin{proof} a) Take a countable covering of X as in Proposition 7.3,
 and consider the exact sequence (4.11). It gives:
$$\dots \ \to \  H^j (X, {\O})  \ \to \  H^j (X, {\H})
  \ \to \  H^{j+1} (X, \R)  \ \to \  H^{j+1} (X, {\O})\  \to \dots .$$
Notice that $H^{j+1} (X, \R)$ is a \u Cech cohomology group,  isomorphic
to the De Rham cohomology group $H_{DR}^{j+1} (X, \R)$, which is Hausdorff. This isomorphism is given by 
a composition of coboundary maps, coming out from the short exact sequences
associated to the sequence 
$$\ 0 \ \to \ \R \ \to \ {\E}^0 \ \to \ {\E}^1 \ \to \ \dots $$
(see f.i. \cite{GH}, p. 44).

By Corollary 7.5,\ $H^{j+1} (X, \R) \simeq H_{DR}^{j+1} (X, \R)$
is a topological isomorphism, so that  $H^{j+1} (X, \R)$ \ is Hausdorff for every $j \geq 0$; when $dim \ H^j (X, \O ) < \infty$,
 by Proposition 7.6 (ii), also $H^j (X, {\H})$ 
is Hausdorff. In our hypotheses, this is true when $j=2q+1$ in case (1) and when $j=2q+2$ in case (2).
\bigskip

b) If $q = 0$, let us recall the exact sequences (4.1) and (4.4):
$$
 0 \to \H \stackrel{j}{\to}  {\E}^{0,0}_\R   \stackrel{i \ddb}{\to}  {\E}^{1,1}_\R \stackrel{d}{\to} ({\E}^{2,1} \oplus {\E}^{1,2})_\R  
 \stackrel{d}{\to} {\E}^{4}(M)_\R \to \dots 
$$
 $$
  0 \to \H \stackrel{\sigma_{-1}}{\to}    ( \Omega^1 \oplus \E^{0,0} \oplus \overline \Omega^1)_{\R} \stackrel{\sigma_{0}}{\to} (\E^{1,0} \oplus \E^{0,1})_{\R} \stackrel{\sigma_{1}}{\to} \E^{1,1}_{\R} \stackrel{\sigma_{2}}{\to} \E^{2,2}_{\R} \stackrel{\sigma_{3}}{\to} ({\E}^{3,2} \oplus {\E}^{2,3})_\R \to \dots
$$
and also recall that
$$H_{\ddb}^{1,1}(X, \R) = \frac{(Ker d)(X)}{(Im\  i \ddb)(X)}, \ \ 
H_{\de + \db}^{1,1}(X, \R) = \frac{(Ker \sigma_{2})(X)}{(Im \sigma_{1})(X)}.$$
Thus we have in the first case
$$
 0 \to \H \stackrel{j}{\to}  {\E}^{0,0}_\R   \stackrel{i \ddb}{\to}  Ker d  \to 0 
$$
 which gives:
$$\ 0 \ \to \ H^0 (X, {\H}) \ \to \ H^0 (X, {\E}^{0,0}_{\R}) \  \stackrel{(i \ddb)_0}{\to}
 \ H^0 (X, Ker \ d) \ \stackrel{\delta_*^0}{\to}  \ H^1 (X, {\H}) \ \to \ 0 \ .$$

Take a suitable covering $\A$ of $X$ as in Proposition 7.3;
by Corollary 7.5, $\delta_*^0$ is a topological homomorphism, and it gives a topological isomorphism 
$$H^1 (X, {\H})   \simeq 
 {{H^0(X, Ker \ d)}\over{Ker \ \delta_*^0}} = 
 {{H^0(X, Ker \ d)}\over{Im \ i \ddb_0}}= H_{\ddb}^{1,1}(X, \R) .$$

Since we know that $dim \ H^1 (X, \O ) < \infty$ we get 
that, if $q=0$, $H_{\ddb}^{q+1,q+1}(X, \R)$ is Hausdorff, as in a).

In the second case, we have:
 $$
  0 \to \H \stackrel{\sigma_{-1}}{\to}    ( \Omega^1 \oplus \E^{0,0} \oplus \overline \Omega^1)_{\R} \stackrel{\sigma_{0}}{\to} Ker \sigma_1 \to 0
$$
 $$
  0 \to  Ker \sigma_1 \to (\E^{1,0} \oplus \E^{0,1})_{\R} \stackrel{\sigma_{1}}{\to} Ker \sigma_2 \to 0 ,$$
so that: 
$$\ \dots \ \to \ H^1 (X, {\H}) \ \to \ H^1 (X, ( \Omega^1 \oplus \E^{0,0} \oplus \overline \Omega^1)_{\R}) \  \to
 \ H^1 (X, Ker \sigma_1) \ \stackrel{\delta_*^1}{\to}  \ H^2 (X, {\H}) \ \to \dots .$$
and 
$$\ 0 \ \to \ H^0 (X, Ker \sigma_1) \ \to \ H^0 (X,  (\E^{1,0} \oplus \E^{0,1})_{\R}) \ \to
 \ H^0 (X, Ker \sigma_2) \ \stackrel{\delta_*^0}{\to}  \ H^1 (X, Ker \sigma_1) \ \to \ 0 \ .$$
From the second sequence, we get as before, using the topological homomorphism $\delta_*^0$, that 
$$H^1 (X, Ker \sigma_1)   \simeq 
 {{H^0(X, Ker \sigma_2)}\over{Ker \ \delta_*^0}} = 
 {{H^0(X, Ker \sigma_2)}\over{H^0(X, Im \sigma_1)}}= H_{\de + \db}^{1,1}(X, \R) .$$
From the first sequence, using Proposition 7.6(ii), we get that 
$H^1 (X, Ker \sigma_1)$ is Hausdorff when 
$dim \ H^2 (X, \O ) < \infty$ and $dim \ H^1 (X, \Omega^1 ) < \infty$, which is precisely our hypothesis.

\medskip

c) If $q > 0$, let us consider first of all $H_{\ddb}^{q+1,q+1}(X, \R)  = \frac{(Ker \sigma_{2q+1})(X)}{(Im \sigma_{2q})(X)}.$

Using the short exact sequences in (4.12), we get as before ($\delta_*^0$ becomes a topological homomorphism):
$$H_{\ddb}^{q+1,q+1}(X, \R) \simeq H^1 (X, {K}er \  \sigma_{2q}) .$$
Repeating this feature back and back, we get a topological isomorphism
$$H_{\ddb}^{q+1,q+1}(X, \R) \simeq H^{q+1} (X, {K}er \  \sigma_{q}).$$
From here on, we have to take in account the sheaves $\Omega^j$.

Consider in (4.12) the first short exact sequence
$$\ 0 \  \to \  {\H} \  \to \  {\L}^{0} \  \to \  {K}er \enskip \sigma_{1} \  \to \  0 \ ,$$
which gives:
$$\ \dots \ \to \ H^{2q} (X, {\L}^{0}) \ \to \ H^{2q} (X, {K}er \  \sigma_1) \ \stackrel{\delta_*^{2q}}{\to}
 \ H^{2q+1} (X, {\H}) \ \to \ \dots \ .$$

Since by the assumption 
$dim \ H^{2q} (X, \Omega^{1} ) < \infty$ and  $dim \ H^{2q+1} (X, \O ) < \infty$, so that  $H^{2q+1} (X, {\H})$ is Hausdorff, by Proposition 7.6
$H^{2q} (X, {K}er \  \sigma_1)$ is Hausdorff.
\smallskip

Using  the second short exact sequence in (4.12), we can prove
 that also $H^{2q-1} (X, {K}er \  \sigma_2)$ is Hausdorff, and so on until 
$H^{q+1} (X, {K}er \  \sigma_{q}) \simeq H_{\ddb}^{q+1,q+1}(X, \R)$, which becomes  Hausdorff.
What is needed at every step is contained in the hypothesis:

$dim \ H^j (X, \Omega^{2q+1-j} ) < \infty \quad \forall j \in \{q+1, \dots, 2q+1 \}$.
\medskip

As for $H_{\de + \db}^{p,p}(X, \R) = \frac{(Ker \sigma_{2p})(X)}{(Im \sigma_{2p-1})(X)},$ when $p:= q+1 > 1$,
use the short exact sequences in (4.12), starting from 
 $$\ 0 \  \to \  {K}er \  \sigma_{2p-1} \  \to \  {\B}^{2p-1} \ \to \  {K}er \  \sigma_{2p} \  \to \  0 \ ,$$ 
which gives:
$$\ 0 \ \to \ H^0 (X, {K}er \  \sigma_{2p-1}) \ \to \ H^0 (X, {\B}^{2p-1}) \ \to
 \ H^0 (X, {K}er \  \sigma_{2p}) \ \stackrel{\delta_*^0}{\to} 
\ H^1 (X, {K}er \  \sigma_{2p-1}) \ \to \ 0 .$$

Since as above $\delta_*^0$ becomes a topological homomorphism, we get 
$$H_{\ddb}^{p,p}(X, \R) \simeq H^1 (X, {K}er \  \sigma_{2p-1}) .$$
Repeating this feature, we get  topological isomorphisms
$$H_{\ddb}^{p,p}(X, \R) \simeq H^1 (X, {K}er \  \sigma_{2p-1})  \simeq H^2 (X, {K}er \  \sigma_{2p-2}) \simeq \dots  \simeq H^{p} (X, {K}er \  \sigma_{p}).$$

On the other hand, consider in (4.12) the first short exact sequence
$$\ 0 \  \to \  {\H} \  \to \  {\L}^{0} \  \to \  {K}er \enskip \sigma_{1} \  \to \  0 \ ,$$
which gives
$$\ \dots \ \to \ H^{2p-1} (X, {\L}^{0}) \ \to \ H^{2p-1} (X, {K}er \  \sigma_1) \ \stackrel{\delta_*^{2p-1}}{\to}
 \ H^{2p} (X, {\H}) \ \to \ H^{2p} (X, {\L}^{0}) \ \to \ \dots \ .$$

Since by the hypothesis, 
$dim \ H^j (X, \Omega^{2p-j} ) < \infty \quad \forall j \in \{p, \dots, 2p \}$,
and  thus $H^{2p} (X, {\H})$ is Hausdorff, then
$H^{2p-1} (X, {K}er \  \sigma_1)$ is Hausdorff by Proposition 7.6.
\smallskip

Using the next exact sequences in (4.12) we get that 
$H^{p} (X, {K}er \ \sigma_{p})$ is Hausdorff, hence we conclude that
$H_{\ddb}^{p,p}(X, \R)$ is Hausdorff.
\end{proof}
\medskip

\begin{prop} Let $X$ be a complex manifold, let $q \geq 0$.
If 
$$dim \ H^j (X, \Omega^{2q+2-j} ) < \infty \quad \forall j \in \{q+2, \dots, 2q+2 \},$$ 
then $W^{q+2,q+1}_{\R} (X) := {{(Ker \  \sigma_{2q+2} ) (X)}\over{(Im \  \sigma_{2q+1} ) (X)}}$ is Hausdorff.

\end{prop}

\begin{proof} The proof is very similar to that of Theorem 7.7, taking in account also the proof of Proposition 4.6.
\end{proof}
\medskip

By comparing Theorem 7.7 and Proposition 7.8 with Proposition 4.6, one gets the following result:

\begin{cor}  Let $X$ be a complex manifold, let $p \geq 1, q := p-1 \geq 0$.
\begin{enumerate} 
\item If $dim \ H^j (X, \Omega^{2p+1-j} ) < \infty \quad \forall j \in \{p+1, \dots, 2p+1 \}$, then $\sigma_{2p}$ is a topological homomorphism, thus
$Im \sigma_{2p}$ is a closed subspace.

\item If  $dim \ H^j (X, \Omega^{2q+2-j} ) < \infty \quad \forall j \in \{q+2, \dots, 2q+2 \}$, then $\sigma_{2q+1}$ is a topological homomorphism, thus $Im \sigma_{2q+1}$ is a closed subspace.

\item If $dim \ H^j (X, \Omega^{2p-j} ) < \infty \quad \forall j \in \{p, \dots, 2p \}$, then $\sigma_{2p-1}$ is a topological homomorphism, thus $Im \sigma_{2p-1}$ is a closed subspace.
\end{enumerate}
\end{cor}

\bigskip
\section{ Duality on non compact manifolds}

For a generic manifold $X$, ${\E}_{p,p}'(X)_{\R} \neq {\D}_{p,p}'(X)_{\R}$; hence to get informations as before about the existence of a suitable $(p,p)-$form, we need to fix a compact $K$ in $X$ as a \lq\lq bound\rq\rq \ for the support of the currents. In this setting, we give
the following list of characterization theorems, whose geometric signification we shall explain with a couple of examples after the proofs.  Notice that we use here transverse forms and strongly positive currents, but we could have chosen also the other notions of positivity. Let us denote the closure of a linear subspace $L$ (in the weak topology) by $(L)^-$.

\begin{thm} Let $X$ be a complex manifold of dimension $n \geq 2$, let $K$ be a compact subset of $X$; let $1 \leq p \leq n-1$ and denote $q := p-1$ in the subscript of the operators. Then:

\begin{enumerate}
\item There  is a  real $(p,p)-$form $\Omega$ on $X$ such that $\Omega \in Im \sigma_{2q}$    and $\Omega_x >0 \ \forall x \in K \ \iff $
there are no non trivial currents $T \in {\E}_{p,p}'(X)_{\R}$, $T \geq 0$, $T \in Ker \sigma_{2q}' ,\  supp T \subseteq K.$ 

\item There  is a  real $(p,p)-$form $\Omega$ on $X$ such that $\Omega \in Im d_{2p-1}$    and $\Omega_x >0 \ \forall x \in K \ \iff$
there are no currents $R \in {\E}_{2p}'(X)_{\R}$, $R \in Ker d_{2p-1}',\  supp R \subseteq K$  with $R_{p,p} := T \geq 0, T \neq 0$.

\item Suppose that $\sigma_{2p-1}$ is a topological homomorphism, so that $Im \sigma_{2p-1}$ is a closed subspace of ${\E}^{p,p}(X)_{\R}$. Then: 

There  is a  real $(p,p)-$form $\Omega$ on $X$ such that $\Omega \in Im \sigma_{2p-1} \cap Ker \sigma_{2q+1}$    and $\Omega_x >0 \ \forall x \in K \ \iff$
there are no non trivial currents $T \in {\E}_{p,p}'(X)_{\R}$, $T \in ((Im \sigma_{2q+1}')^- + Ker \sigma_{2p-1}')^-,\ $  $T \geq 0$,  $supp T \subseteq K.$

\item There  is a  real $(p,p)-$form $\Omega$ on $X$ such that $\Omega \in Im \sigma_{2p-1}$    and $\Omega_x >0 \ \forall x \in K \ \iff $
there are no non trivial currents $T \in {\E}_{p,p}'(X)_{\R}$, $T \geq 0$, $T \in Ker \sigma_{2p-1}'$ and $ supp T \subseteq K  $.

\end{enumerate}
\end{thm}

\begin{proof} In all cases, one part of the proof is simple: if there exists the form $\Omega$ and also the current $T$ (or $R$) as given in the Theorem, we would have in cases (1), (2) and (4):

$(T, \Omega)=0$ (or $(R, \Omega)=0$) because ($R$ or) $T \in Ker L' = (Im L)^{\perp}$  for some operator $L$.

In case (3), $T \in ((Im \sigma_{2q+1}')^- + Ker \sigma_{2p-1}')^-$, that is, 
$T = \lim_{\varepsilon} (T_{\varepsilon}' + T_{\varepsilon}'')$ with $T_{\varepsilon}'' \in Ker \sigma_{2p-1}', T_{\varepsilon}' = 
\lim_{\mu} T_{\varepsilon, \mu}' $ with $T_{\varepsilon, \mu}' \in Im \sigma_{2q+1}'.$

Hence $(T, \Omega)= \lim_{\varepsilon} (T_{\varepsilon}', \Omega) + \lim_{\varepsilon} (T_{\varepsilon}'', \Omega)$; the first addendum vanishes because $\Omega \in Ker \sigma_{2q+1}$, the second one because $\Omega \in  Im \sigma_{2p-1}$.

Moreover, $(T, \Omega)= (\chi_K T, \Omega)> 0$, since $\chi_K T \geq 0$ does not vanish and $\Omega_x >0 \ \forall x \in K$. 
\smallskip

Let us prove now the converses.
\smallskip

Case (4). Consider on $X$ a hermitian metric $h$ with associated $(1,1)-$form $\gamma$, and let
$$A= \{ \Theta \in {\E}^{p,p}(X)_{\R} / \exists c > 0 \ {\rm such \  that} \ \Theta_x > c \gamma_x^p \ \forall x \in K \}$$
(the condition obviously means that $\Theta_x - c \gamma_x^p$ is strictly weakly positive).

It is easy to control that $A$ is a non empty convex open subset in the topological vector space ${\E}^{p,p}(X)_{\R}$. If there is no form $\Omega$ as stated in the Theorem, then $A \cap Im \sigma_{2p-1} = \emptyset$, where $Im \sigma_{2p-1}$ is a linear subspace in ${\E}^{p,p}(X)_{\R}$.

By the Hahn-Banach Theorem 4.4, we get a separating closed hyperplane, which is nothing but a current $T \in {\E}_{p,p}'(X)_{\R}$, for which we can suppose:

$T \in (Im \sigma_{2p-1})^{\perp} = Ker \sigma_{2p-1}'$,

$(T, \Theta)>0$ for every $\Theta \in A$ (thus $T \neq 0$).

Let us check that $T \geq 0$, i.e., by Definition 2.4, that $(T, \Omega) \geq 0, \ \forall \Omega \in WP^{p,p}(X)$. For every $\varepsilon >0$, $\Omega + \varepsilon \gamma^p \in A$, thus 
$$(T, \Omega) = (T, \lim_{\varepsilon \to 0} (\Omega + \varepsilon \gamma^p)) = \lim_{\varepsilon \to 0} (T, \Omega + \varepsilon \gamma^p) \geq 0.$$
Moreover, $supp T \subseteq K$; indeed, let $\alpha \in {\E}^{p,p}(X)_{\R}$ with $supp \ \alpha \subseteq X-K$; then for every $t \in \R$, it holds 
$t \alpha +  \gamma^p \in A$. Therefore $0 < (T,  t \alpha +  \gamma^p) = t(T, \alpha) +  (T, \gamma^p)$: this is not possible for every $t \in \R$, until $(T, \alpha)=0$, as required.
\smallskip

Case (1) is very similar, it is enough to replace  $\sigma_{2p-1}$ by $\sigma_{2q}$. This result was proved by Theorem 3.2(i) in \cite{ABL}.
\smallskip

Case (3). We can proceed as above, replacing  $Im \sigma_{2p-1}$ by $Im \sigma_{2p-1} \cap Ker \sigma_{2q+1}$,which is a linear subspace in ${\E}^{p,p}(X)_{\R}$.

Thus we get $T \in (Im \sigma_{2p-1} \cap Ker \sigma_{2q+1})^{\perp}$: here we use the closure of $Im \sigma_{2p-1}$ to go further (see f.i. \cite{Sc} p. 127), so we get, as required,
$$T \in (Im \sigma_{2p-1} \cap Ker \sigma_{2q+1})^{\perp} = ((Im \sigma_{2p-1})^{\perp} + (Ker \sigma_{2q+1})^{\perp})^- =  ((Im \sigma_{2q+1}')^- + Ker \sigma_{2p-1}')^-.$$
Case (2). 
Let us consider the l.c.s. ${\E}_{2p}'(X)_{\R} = ( \oplus_{a+b=2p} \E_{a,b}'(X))_\R $, and denote by $\pi$ the projection on the addendum ${\E}_{p,p}'(X)_{\R}$. Notice that $\pi(Ker d_{2p-1}')$ is a closed convex non-empty subset of ${\E}_{p,p}'(X)_{\R}$.

Take
$$P_K(X)= \{ T \in {\E}_{p,p}'(X)_{\R} / supp T \subseteq K, T \geq 0  \}$$
which is a closed convex cone with a compact basis given by  
$$\tilde P_K(X) = \{ T \in P / (T, \gamma^p) =1  \}.$$
By our hypothesis, $\pi(Ker d_{2p-1}') \cap \tilde P_K(X) = \emptyset$.

Using the Separation Theorem 4.5, we get a closed hyperplane in ${\E}_{p,p}'(X)_{\R}$, strictly separating $\pi(Ker d_{2p-1}')$ and $\tilde P_K(X)$; hence we get 
$\Omega \in {\E}^{p,p}(X)_{\R}$ such that $(T, \Omega)>0$ for every $T \in \tilde P_K(X)$, and $\Omega \in (\pi(Ker d_{2p-1}'))^{\perp}$.

The first condition assures that $\Omega_x >0 \ \forall x \in K$. In fact, by Proposition 2.2 we have to check that
 $\Omega_x(\sigma_p^{-1} V \wedge \overline V) > 0$
for every $V \in \Lambda_{p,0}(T_x'X)$, $V \neq 0$ and simple. But given such a vector, the Dirac current $T := \delta_x (\sigma_p^{-1} V \wedge \overline V) \in P_K(X)$, so that for some $c > 0$, $cT \in \tilde P_K(X)$ and thus $\Omega(\sigma_p^{-1} V \wedge \overline V) = T(\Omega) > 0$

The second one implies that, for every current $R \in Ker d_{2p-1}'$, $(R, \Omega) = (\pi(R), \Omega)=0$ since  $\Omega \in {\E}^{p,p}(X)_{\R}$. Thus $\Omega \in (Ker d_{2p-1}')^{\perp} = (Im d_{2p-1})^- = Im d_{2p-1}$, because $d$ is always a topological homomorphism.
\end{proof}

\medskip

{\bf 8.1.1 Remark.} Notice that, as before,
$8.1(1)  \Longrightarrow  8.1(2)  \Longrightarrow  8.1(3)  \Longrightarrow  8.1(4).$
The stronger condition, 8.1(1), is in fact a sort of \lq\lq local\rq\rq $p-$K\"ahler condition with exact (that means $\ddb-$exact) form.
\medskip

{\bf 8.1.2 Remark.} Every $n-$dimensional connected non compact manifold is $n-$complete; thus it is $nK$ with an exact form.

\medskip

In particular, when $M$ is a compact manifold, the above statement get simplified, as  Theorem 5.1 showed.

\medskip

And finally, let us consider an analogue of Theorem 3.2 for non compact manifolds.

\begin{thm} 
Let $X$ be a complex manifold of dimension $n \geq 2$, let $K$ be a compact subset of $X$; let $1 \leq p \leq n-1$ and denote $q := p-1$ in the subscript of the operators. 

\begin{enumerate}

\item Suppose $\sigma_{2q+1}$ is a topological homomorphism.Then:

there is a real $(p,p)-$form $\Omega$ on $M$ such that $\Omega \in Ker \sigma_{2q+1}$    and $\Omega_x >0 \ \forall x \in K \ \iff$ there are no non trivial currents $T \in {\E}_{p,p}'(X)_{\R}$, $T \geq 0$, $T \in Im \sigma_{2q+1}',$ and $supp T \subseteq K .$ 

\item Suppose $\sigma_{2p-1}$ and $\sigma_{2q+1}$ are topological homomorphisms, and that $Im \sigma_{2p-1}$ has finite codimension in $Ker \sigma_{2p}$.Then:

there is a real $(p,p)-$form $\Omega$ on $M$ such that $\Omega \in Ker \sigma_{2q+1} + Im \sigma_{2p-1}$    and $\Omega_x >0 \ \forall x \in K \ \iff$ there are no non trivial currents $T \in {\E}_{p,p}'(X)_{\R}$, $T \geq 0$, $T \in Ker \sigma_{2p-1}' \cap Im \sigma_{2q+1}',\  $ $supp T \subseteq K .$ 

\item There is a real $2p-$form $\Psi$ with $\Psi^{p,p} := \Omega$
 on $M$ such that $\Psi \in Ker d_{2p}$    and $\Omega_x >0 \ \forall x \in K \ \iff$
there are no non trivial currents $T \in {\E}_{p,p}'(X)_{\R}$, $T \geq 0$, $T \in Im d_{2p}',\  supp T \subseteq K .$ 

\item Suppose $\sigma_{2p}$ is a topological homomorphism.Then:

there is a  real $(p,p)-$form $\Omega$ on $M$ such that $\Omega \in Ker \sigma_{2p}$    and $\Omega_x >0 \ \forall x \in K \ \iff$ there are no non trivial currents $T \in {\E}_{p,p}'(X)_{\R}$, $T \geq 0$, $T \in Im \sigma_{2p}'$ and $ supp T \subseteq K .$

\end{enumerate}
\end{thm}

\begin{proof}
As seen in the previous theorems, 
in all cases, one part of the proof is simple, and does not require the hypotheses on topological homomorphims: if there exists the form $\Omega$ and also the current $T$ as given in the Theorem, we would have 
$(T, \Omega)> 0$ on $K$. 

But, in case (1), 
$(T, \Omega)=0$  because $\Omega \in Ker \sigma_{2q+1}$ and $T \in Im \sigma_{2q+1}' \subseteq (Ker \sigma_{2q+1})^{\perp}$.
The same holds in the other cases.

\medskip

For the converses, we go on as in the proof of Theorem 8.1. Let us sketch here only major changes.

Case (1). Consider the non empty convex open set  $A \in {\E}^{p,p}(X)_{\R}$; if no \lq\lq right\rq\rq form $\Omega$ exists, we get  $A \cap Ker \sigma_{2q+1} = \emptyset$,
thus there is a current $T \in {\E}_{p,p}'(X)_{\R}$, with
$T \in (Ker \sigma_{2q+1})^{\perp} = Im \sigma_{2q+1}'$, by the hypothesis, and $(T, \Theta)>0$ for every $\Theta \in A$.
This gives $T \geq 0$, $T \neq 0$ and $supp T \subseteq K$, as seen in the proof of Theorem 8.1. 

Case (1) was proved in \cite{ABL}, Theorem 3.2(ii).
 
It is the same, more or less, in cases (2) and (4).

In case (3), we separate $\tilde P_K(X)$, the compact basis of $P_K(X)= \{ T \in {\E}_{p,p}'(X)_{\R} / supp T \subseteq K, T \geq 0  \}$, from the closed convex set $\pi(Im d_{2p}')$
(notice that $Im d_{2p}'$ is closed because the operator $d$ is always a topological homomorphism).

Hence we get $\Omega \in (\pi(Im d_{2p}'))^{\perp}$. But $(\Omega, R) = (\Omega, \pi(R))$, since $\Omega$ has bidegree $(p,p)$, thus 
$\Omega \in (Im d_{2p}')^{\perp} = Ker d_{2p}$.
\end{proof}

\medskip

{\bf 8.2.1 Remark.} Notice that, as before,
$8.2(1)  \Longrightarrow  8.2(2)  \Longrightarrow  8.2(3)  \Longrightarrow  8.3(4).$
\medskip

{\bf 8.2.2 Remark.} In \cite{A2} we use closed real $(p,p)-$forms, which are positive on a fixed compact set, to give the definition of locally \pkk manifold (see \cite{A2}, Definition 6.1) and to study when, in a proper modification $f: \tilde X \to X$ with compact center, the property of being locally $(n-1)-$K\"ahler comes back from $X$ to $\tilde X$. 
\medskip

Let us give a simple application of Theorem 8.1.

Suppose $X$ has a compact (irreducible) analytic subspace $Y$ of dimension $m \geq 1$. Then $T := [Y]$ is a \lq\lq closed\rq\rq positive non-vanishing current of bidimension $(m,m)$ with $supp \ T = [Y]$.
Thus there are no \lq\lq exact\rq\rq $(m,m)-$forms on $X$ with $\Omega > 0$ on $Y$. 

Nevertheless, there are \lq\lq exact\rq\rq $(p,p)-$forms  on $X$ with $\Omega > 0$ on $Y$ for every $p > m$: in fact, if not, by Theorem 8.1(1) there would exist a pluriharmonic
 (i.e. $\ddb -$closed)  positive current of bidimension $(p,p)$, supported on $Y$, whose dimension is to small: hence $T=0$ (see f.i. \cite{AB4}, Theorem 1.2).

\medskip
We end this paper showing how one can obtain the results we cited in Theorem 6.4 about 1-convex manifolds using a sort of  \pkk form as given il Theorem 8.1. Let us prove only the following result:

\lq\lq Let $X$ be a 1-convex manifold with exceptional set $S$ of dimension $k$. Then $X$ is $p-$K\"ahler for every $p > k$, with a $\ddb$-exact $p-$K\"ahler form.\rq\rq

\begin{proof} Let $f: X \to Y$ be the Remmert reduction of $X$ (see Remark 6.1.1); 
$Y$ is embeddable in $\C^n$, hence it carries a K\"ahler form $\omega ' = i \ddb g$. 
Let $\omega  := f^* \omega '$; $\omega$ is positive on $X$ and transverse on $X - S$.

Consider a compactly supported current $T \in {\E}_{p,p}'(X)_{\R}$, $T \geq 0$, $T \in Ker \sigma_{2q}'$, i.e. $i \ddb T = 0$, as in Theorem 8.1(1). Since $\omega ^p\in Im \sigma_{2q}$, $T(\omega ^p) =0$, so that ${\rm supp} T \subseteq S$. By Theorem 1.2 in \cite{AB4}, $T=0$, because $p > k$.

Thus by Theorem 8.1(1) we get a  real $(p,p)-$form $\Omega$ on $X$ such that $\Omega \in Im \sigma_{2q}$, i.e. $\Omega = i \ddb \theta$,    and $\Omega_x >0 \ \forall x \in S$.

Take a compactly supported smooth function $\chi$ such that $0 \leq \chi \leq 1, \ \chi =1$ on $S$. For $C >>0$, $C\omega^p + i \ddb (\chi \theta)$ is a 
$\ddb -$exact real form, which is transverse on the whole of $X$.
\end{proof}

\end{document}